\documentclass[12pt,leqno]{amsart}
\usepackage {enumerate}
\usepackage{amssymb}
\usepackage{amsmath}
\usepackage{amsthm}
\usepackage{pstricks, pst-poly}
\usepackage{amsbsy}
\usepackage{bm}
\usepackage{graphicx}%
\usepackage{color}
\usepackage{hyperref}
\usepackage{pstricks-add}
\theoremstyle{plain}
\newtheorem{theorem}{Theorem}[section]
\newtheorem{lemma}[theorem]{Lemma}
\newtheorem{corollary}[theorem]{Corollary}
\newtheorem{proposition}[theorem]{Proposition}

\theoremstyle{definition}
\newtheorem{definition}[theorem]{Definition}

\theoremstyle{definition}
\newtheorem{notation}[theorem]{Notation}

\newtheorem{remark}[theorem]{Remark}

\def\@enum@{\list{\csname label\@enumctr\endcsname}%
           {\usecounter{\@enumctr}\def\makelabel##1{\hss\llap{##1}}
           \itemsep=2pt\parsep=0pt\topsep=3pt plus 1pt minus 1 pt}}

\newenvironment{numlist}{\enumerate[(1)]}{\endenumerate}

\newenvironment{proclama}[1]{
                \par\vspace{\topsep}\noindent{\bf #1}
                \begin{em}}
                {\end{em}\par\vspace{\topsep}}

 \newenvironment{proclama not emphasized}[1]{\par\vspace{\topsep}\noindent{\bf #1}}{\par\vspace{\topsep}}

\newcommand{\Z}{\mathbb{Z}}
\newcommand{\R}{\mathbb{R}}

\renewcommand{\H}{{\mathbb{H}}}
\newcommand{\SI}{ {\mathbb{S} }^1}

\newcommand{\inu}{\iota}

\newcommand{\sm}{\setminus}

\newcommand{\cc}{\mathcal{C}}

\newcommand{\la}{\langle}
\newcommand{\ra}{\rangle}

\newcommand{\geo}{\Gamma}
\newcommand{\geop}{\Gamma'}

\newcommand{\isom}{\mathrm{Isom}(\H)}

\newcommand{\tg}{{\gamma}}

\newcommand{\conjj}[1]{\langle{#1} \rangle}

\newcommand\map[3]{\mbox{${#1}\colon {#2} \longrightarrow {#3}$}}

\newcommand{\xx}{x_1}
\newcommand{\yy}{y_1}

\newcommand{\tax}{\tau_x}
\newcommand{\tay}{\tau_y}
\newcommand{\taxx}[1]{\tau_{#1}}
\newcommand{\taxy}{\tau_{xy}}
\newcommand{\fhc}[1]{\widetilde{#1}}

\newcommand{\pri}{\Theta}
\newcommand{\pro}{\Pi}
\newcommand{\clo}{\delta}

\begin{document} 

\title[{The extended Goldman bracket determines intersection numbers for orbifolds}]{The Goldman bracket determines intersection numbers for  surfaces and orbifolds}

\author{Moira Chas}

\address{	Department of Mathematics,\\
		Stony Brook University\\
		Stony Brook, NY, 11794}

\email{moira@math.sunysb.edu}

\author{Siddhartha Gadgil}

\address{	Department of Mathematics,\\ 
		Indian Institute of Science,\\
		Bangalore 560012, India}

\email{gadgil@math.iisc.ernet.in}

\date{\today}


\thanks{Supported by NSF grant 1071448-1-46676}


\begin{abstract} 
In the mid eighties Goldman proved an embedded curve could be isotoped to not intersect a closed geodesic if and only if their Lie bracket (as defined in that work) vanished. Goldman asked for a topological proof and about extensions of the conclusion to curves with self-intersection. Turaev, in the late eighties, asked about characterizing simple closed curves algebraically, in terms of the same Lie structure. We show how the Goldman bracket answers these questions for all finite type surfaces. In fact we count self-intersection numbers and mutual intersection numbers for all finite type orientable orbifolds in terms of a new Lie bracket operation, extending Goldman's. The arguments are purely topological, or based on elementary ideas from hyperbolic geometry.

These results are intended to be used to recognize hyperbolic and Seifert vertices and the gluing graph in the geometrization of three manifolds. The recognition is based on the structure of the String Topology bracket of three manifolds. 
\end{abstract}

\maketitle
\begin{center}
\textit{This work is dedicated with deep and grateful admiration to Bill Thurston (1946-2012).}
\end{center}
\section{Introduction}

Goldman \cite{Gol} discovered in the eighties an intriguing  Lie algebra structure on the free abelian group generated by the set of free homotopy classes of closed directed curves on an oriented surface $F$. The definition of the Goldman bracket  combines intersection structure with the usual based loop product in the following way: 
Given two closed  free homotopy classes $a$ and $b$ with representatives  $A$ and $B$ respectively, intersecting only in transversal double points,  
\begin{equation}\label{bracket def}
[a,b]=\sum_{P\in A \cap B}\mathrm{sign}(P)  \fhc{ A\cdot_P B}
\end{equation}
where $\mathrm{sign}(p)$  is the sign of the intersection between the curves $A$  and $B$ at $P$ and $A\cdot_p B$ is the loop product of $A$ and $B$ both viewed as based at $P$,  and $\fhc{ C }$ is the free homotopy  class of a curve $C$. This bracket is extended by linearity to the free module generated by free homotopy classes of curves.
Goldman showed that this bracket is well defined, skew-symmetric and satisfies the Jacobi identity.

Clearly, if $a$ and $b$ are free homotopy classes that have disjoint representatives, then   $[a,b]$ is zero. Goldman \cite{Gol} also  showed (using Thurston's earthquakes) that this bracket has the remarkable property that if one of the classes, $a$ or $b$ has a simple representative, then the bracket $[a,b]$ vanishes if and only if $a$ and $b$ can be represented by disjoint curves. Goldman asked for a topological proof and about extensions of the conclusion to curves with self-intersection. Turaev, in the late eighties, asked about characterizing simple closed curves algebraically in terms of this Lie structure. 

Later on  Chas \cite{chas2} proved that if either $a$ or $b$ has a simple representative then the bracket of $a$ and $b$ counts the geometric  intersection number between $a$ and $b$ (by geometric intersection number we mean  the minimum number of points, counted with multiplicity, in which representatives of $a$ and $b$ intersect).

On the other hand, there are examples of classes $a$ and $b$ with no disjoint representatives and such that $[a,b]=0$ \cite[Example 9.1]{chas}.  The bracket is a homotopy invariant analogous to the set of conjugacy classes in the fundamental group and it is, in some sense, simpler than the fundamental group itself.
Since intersection and self-intersection numbers of closed curves on surfaces play  such a critical  role in several areas of low-dimensional topology, it is highly desirable to find such characterizations of the intersection number. A result of this nature, obtained by Chas and Krongold~\cite{CK}, was that for the subset of compact orientable surfaces with non-empty boundary, the  bracket $[a,a^3]$ determines the self-intersection number of $a$. 

Finally, after the twenty five years since Goldman's  paper \cite{Gol} we show  here how the  bracket answers the question about disjunction and simplicity of closed curves for all finite type  surfaces. We count self-intersection numbers and mutual intersection numbers for all finite type orientable orbifolds in terms of a new Lie bracket operation,  extending Goldman's. Our results fill in most of the lacunae in partial results that have resisted extension over the intervening years. The arguments are purely topological, using group theory ideas of Freedman, Scott and Haas  \cite{Scott} and \cite{FHS}, or they are based on  elementary  geometrical ideas from  hyperbolic geometry. 
 
By a \emph{Fuchsian group} we mean a  discrete  group of orientation preserving isometries of the hyperbolic plane. Below are the two main results of this paper. 

\begin{proclama}{Mutual Intersection Theorem} Let $x$ and $y$ by non-conjugate hyperbolic elements in a finitely generated Fuchsian group.
Then the geometric intersection number of $x$ and $y$ is given by  the number of terms (counted with multiplicity) divided by $p.q$, of  $[\conjj{x^p}, \conjj{y^q}]$  for   all but finitely many values of positive integers $p$ and $q$,  satisfying the ratio of the translation length of $x$ by the translation length of $y$ is  not $q/p$ \end{proclama}

\begin{proclama}{Self Intersection  Theorem } For $x$ a hyperbolic element in a finitely generated Fuchsian group, which is not a proper power of another element, the geometric self-intersection number of $x$ is given by the number of terms (counted with multiplicity) divided by $p.q$ of 
$[\la x^p \ra,\la x^q\ra]$ 
for all  but finitely many values of positive integers $p$ and $q$ satisfying $p \ne q$.
\end{proclama}

Our proof is based on the word hyperbolicity of Fuchsian groups  rather than small-cancellation theory as in~\cite{CK}. By extending the result of \cite{chas2} for surfaces with boundary to closed surfaces we complete  the answer to Goldman's question \cite[Subsection 5.17]{Gol}, whether his topological result (if $a$ and $b$ are two free homotopy classes of curves on a surface such that $a$ has a simple representative and $[a,b]=0$, then $a$ and $b$ have disjoint representatives) had a topological proof.

The main lemma of this work states that if at least one of $p$ and $q$ is sufficiently large and the lengths of $x^p$ and $y^q$ are different, then there is no cancellation of terms in the bracket $[\conjj{ x^p}, \conjj{y^q}]$. In other words, if the  representatives $A$ and $B$ intersect in the minimum number of points, then two intersection points $P$ and $Q$ with different sign do not give the same free  homotopy class of curves, that is $\fhc{A \cdot_P B}\ne \fhc{A\cdot_Q B}$. 

We show this by constructing quasi-geodesic representatives of a lift of a loop representing $A \cdot_P B$. These quasi-geodesics are the concatenations of certain segments of  translates of the axis of $x$ and the axis of $y$. 
As quasi-geodesics are not too far from  geodesics, it follows that if two points of intersection give the same free homotopy class, then there is a pair of corresponding quasi-geodesics that  are close, which then implies that they are equal. We deduce that the two points correspond to terms with the same sign in the Goldman bracket.

These  results are intended to be applied to recognize hyperbolic and Seifert vertices and the gluing graph in the geometrization of three manifolds. The recognition is  based on the structure of the String  Topology bracket of three manifolds.

For a typical irreducible three manifold, the cyclic homology of the group ring of the fundamental group  lives in two degrees: zero and one. Degree one is a Lie algebra and degree zero is a Lie module for degree one. The Lie algebra breaks into a direct sum corresponding to the pieces and the module structure tells how they are combined in the graph.

One can show that the Goldman bracket on the linear space with basis the set of free homotopy classes and the power operations on this basis determine the Fuchsian group of an orbifold. Thus, the Goldman bracket solves  "the recognition problem" for two dimensional orbifolds. More significantly, now that the proof of the Geometrization conjecture has enabled  a classification of three manifolds, there arises the need to  calculate the geometrization in examples like knots, i.e. "the recognition problem for three manifolds". Our work directly applies to that since the String Topology bracket in three manifolds will be used to describe the canonical graph of the geometrization picture as well as which vertices are hyperbolic and which are Seifert Fibered spaces. 
 This bracket is largely concentrated on the Seifert pieces.  On these pieces it depends on the  orbifold bracket defined here. The orbifold part of the story seemed sufficiently interesting to present independently with the details of the application to three manifolds coming next.

We emphasize though that the above characterization is a new one for closed curves on closed surfaces, and should be of interest even in this case.

Others have considered String Topology operations  for orbifolds and manifold stacks in a more abstract setting \cite{ABU},  \cite{BGN}, \cite{LUX}. It would be interesting to relate those results to the concrete results here.

This paper benefitted from conversations with Ian Agol, Danny Calegari and Dennis Sullivan. It started when the first author was visiting the Indian Institute of Science in Bangalore, India, to which she would like to express gratitude. Finally, in the final stage of this paper, the authors learned the sad news of Bill Thurston's death. This work wouldn't have been possible without the many directions he opened up in mathematics.

In Section \ref{in} we review the group theoretic definition of intersection number from \cite{FHS} and \cite{Scott} as well as the definition of the geometric intersection number of closed curves in a two dimensional, orientable orbifold.  Section \ref{orbi} is devoted to the extension of the Goldman bracket to oriented orbifolds (a crucial part of this definition is the elementary geometric fact that if two hyperbolic transformations $x$ and $y$ have intersecting axes, then $xy$ is hyperbolic).  In Section~\ref{Jacobi} we prove the Jacobi identity for the extension of the Goldman bracket (interestingly enough, this proof boils down to the proposition of geometry that if a line intersects a side of a triangle, then it intersects one of the other two sides).
In Section~\ref{examples} we give examples of the bracket in the modular surface (a beautiful and computable example of orbifolds)
In Section \ref{dog} we show that geodesics are quantitatively separated for hyperbolic surfaces (and orbifolds). Namely if two closed geodesics sufficiently close and parallel after lifting to the universal cover, they must coincide. 
In Section \ref{ncl} we prove the main non-cancellation lemma, stating that if the conjugacy classes of the two terms of the bracket coincide, then the  
two quasi-geodesic associated to these two terms coincide.  
Finally in Section  \ref{main} we give the proofs of  the Intersection Theorem and the Self -intersection Theorem.

\section{The geometric intersection number and the group theoretic intersection number}\label{in}

Let  $G$ be  a discrete subgroup of orientation preserving isometries of the hyperbolic plane $\H$.  (The set of  isometries  of $\H$, $\isom$  has the compact-open topology.)

Each isometry $g$ of the hyperbolic plane extends to the circle at infinity, where, if $g \ne 1$, it fixes at most $2$ points. An isometry is called \emph{elliptic, parabolic} or \emph{hyperbolic} according as it fixes $0$, $1$ or $2$ points respectively in the circle at infinity. 
A hyperbolic isometry $g$ fixes the (hyperbolic) line joining its two fixed points at infinity. This line is called the \emph{axis of $g$}.  Further, the sets of fixed points at infinity of two isometries contained in a discrete subgroup $G$ are either disjoint or coincide. If the sets of fixed points at infinity of a pair of elements of $G$ coincide and are non-empty, then the isometries are both powers of the same element of $G$.

In this paper, an \emph{orbifold $\H/G$} is the  quotient of the hyperbolic plane $\H$ by  a discrete group of orientation preserving isometries $G$, provided with the induced metric. The pertinent finer notion  of free homotopy for orbifolds is described in Subsection \ref{orbifold}. (Note that we are using the word "orbifold" in a narrower sense than the usual).

In this section we review the definition of closed curves, homotopy and geometric intersection number for curves   for an orbifold   (Subsection \ref{orbifold}), the group theoretic definition of intersection number in orbifolds  (Subsection \ref{labeling}) and show these two definitions agree. 
(The reader is referred to \cite[Chapter 13]{Thurston}, \cite[Chapter 2]{bmp} and \cite[Section 6.2]{K} for a more general definition of  orbifolds and orbifold homotopy . See also \cite[Section 13.3]{Ra} for  a formidable discussion of based orbifold homotopy in terms of charts.)

\subsection{Orbifold homotopy and the geometric intersection number}\label{orbifold}

A \emph{cone point} $P$ in $\H/G$ is the projection of a point in $\H$ which is fixed by some non-trivial element  of $G$. The \emph{order} of a cone point $P$ is the cardinality of the maximal cyclic subgroup of $G$ fixing $P$.

Consider the projection map,  $\map{\pro}{\H}{\H/G}$. A \emph{representative of a closed oriented curve in an orbifold $\H/G$} is a continuous map $\map{\alpha}{\SI}{\H/G}$ ($\H/G$ thought as a topological space), passing through at most finitely many cone points, together with the choice of a full lift $\map{\hat \alpha}{\R}{\H}$, (so that $\pro \circ \hat \alpha = \alpha \circ \pri$, where $\map{\pri}{\R}{\R/2\pi\Z}$ is the usual projection.) Two representatives of closed curves are equivalent if their full lifts are related by an element of the group $G$. A \emph{closed curve on the orbifold $\H/G$} is an equivalence class of representatives of closed curves.

\begin{definition}\label{homotopy} Two closed oriented curves $\alpha$ and $\alpha'$ in $\H/G$ are \emph{$\H/G$-homotopic} if they are related by a finite sequences of moves. Each of these moves is either a homotopy in the complement of  the cone points or is one of the skein relations or moves depicted in Figures~\ref{skein even} and \ref{skein odd}. There, the disk where the move happens contains exactly one cone point $P$, and $n$ denotes the order of $P$.  An arc with  no self-intersection in the disk and passing through $P$ is $\H/G$-homotopic relative to endpoints to an arc  spiraling around $P$ in either direction  $(n-1)/2$ times if $n$ is odd (Figure~\ref{skein odd}), $n/2$ times if $n$ is even (Figure~\ref{skein even}). Also, if $n$ is odd, the endpoints of the arc are antipodal and if $n$ is even, the endpoints coincide.

\begin{figure}[htbp]
\begin{pspicture}(16,6.5)
%
\rput(1,1.5){$\H/G$}
\rput(1,5){$\H$}
\rput(6,5){$\sim$}
\rput(10,5){$\sim$}
\rput(6,1.5){$\sim$}
\rput(10,1.5){$\sim$}
\pscircle(4,5){1.5}
\psdot(4,5)

\psecurve[arrowsize=0.2,showpoints=false]{->}(2,4)(2.5,5)(4,5.64)(5.5,5)(6,4)

\rput(4,5){
\PstPolygon[PolyNbSides =4,PolyIntermediatePoint =0.,PolyRotation=180, unit	= 1.5,linecolor=lightgray]

}

\pscircle(8,5){1.5}
\rput(4,0){
\psdot(4,5)}
\rput(4,0){
\psecurve[showpoints=false,arrowsize=0.2]{->}(2,4)(2.5,5)(4,5)(5.5,5)(6,6)
\psdot(4,5)}

\rput(8,5){
\PstPolygon[PolyNbSides =4,PolyIntermediatePoint =0.,PolyRotation=180, unit	= 1.5,linecolor=lightgray]
}

\pscircle(12,5){1.5}
\rput(8,0){
\psdot(4,5)
}
\rput(8,0){\psecurve[arrowsize=0.2,showpoints=false]{->}(2,6)(2.5,5)(4,4.4)(5.5,5)(6,6)}
\rput(12,5){
\PstPolygon[PolyNbSides=4,PolyIntermediatePoint =0.,PolyRotation=180, unit	= 1.5,linecolor=lightgray]
}

\rput(0,-3.5){

\pscircle(4,5){1.5}
\psdot(4,5)
\psecurve[arrowsize=0.2,showpoints=false]{->}(2,4)(2.5,5)(4,5.64)(5,5)(4,4.4)(3.4,5)(4,5.2)(4.4,5)(2.5,5)(2,6)

\pscircle(8,5){1.5}
\rput(4,0){\psecurve[arrowsize=0.2]{->}(2,4)(2.5,5)(4,5)(2.5,5)(2,6)
\psdot(4,5)}

\pscircle(12,5){1.5}
\rput(8,0){\psecurve[showpoints=false,arrowsize=0.2]{<-}(2,4)(2.5,5)(4,5.64)(5,5)(4,4.4)(3.4,5)(4,5.2)(4.4,5)(2.5,5)(2,6)
\psdot(4,5)
}

}

 \end{pspicture}
 \caption{Skein relations for points order $n=4$ (lower) and the corresponding lifts (upper).}\label{skein even}
\end{figure}
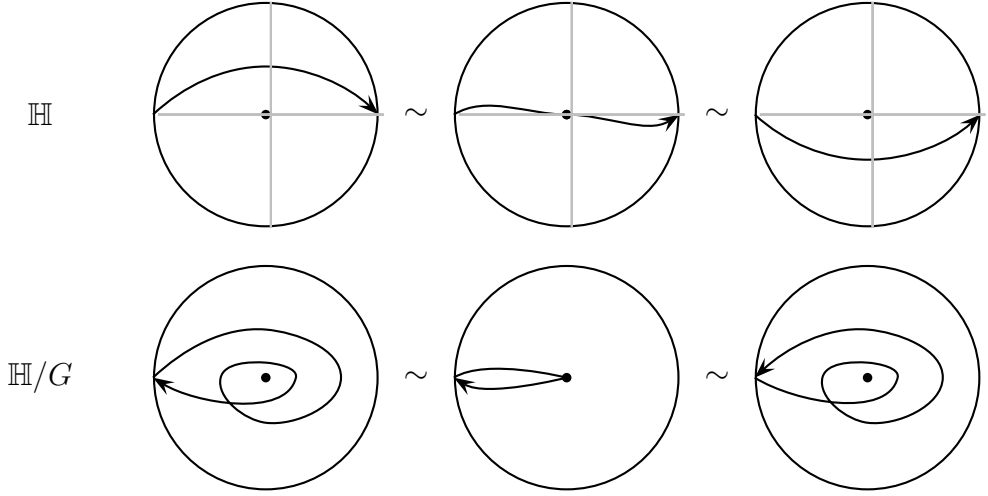

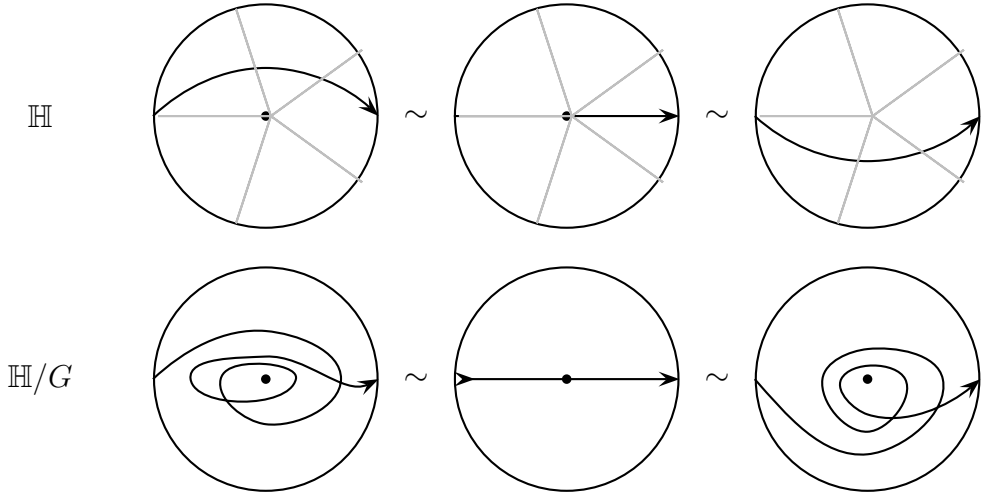
\begin{figure}[htbp]
\begin{pspicture}(16,6.5)

\rput(1,1.5){$\H/G$}
\rput(1,5){$\H$}

\rput(6,5){$\sim$}
\rput(10,5){$\sim$}
\rput(6,1.5){$\sim$}
\rput(10,1.5){$\sim$}

\pscircle(4,5){1.5}
\psdot(4,5)
\psecurve[arrowsize=0.2,showpoints=false]{->}(2,4)(2.5,5)(4,5.64)(5.5,5)(6,4)

\rput(4,5){
\PstPolygon[PolyIntermediatePoint =0.,PolyRotation=180, unit	= 1.5,linecolor=lightgray]
}

 \pscircle(8,5){1.5}
\rput(4,0){\psline[arrowsize=0.2]{->}(2.5,5)(5.5,5)
\psdot(4,5)}

\rput(8,5){
\PstPolygon[PolyIntermediatePoint =0.,PolyRotation=180, unit	= 1.5,linecolor=lightgray]
}

\pscircle(12,5){1.5}
\rput(8,0){\psecurve[arrowsize=0.2,showpoints=false]{->}(2,6)(2.5,5)(4,4.4)(5.5,5)(6,6)}
\rput(12,5){
\PstPolygon[PolyIntermediatePoint =0.,PolyRotation=180, unit	= 1.5,linecolor=lightgray]
}

\pscircle(4,1.5){1.5}
\rput(0,-3.5){
\psdot(4,5)
\psecurve[arrowsize=0.2,showpoints=false]{->}(2,4)(2.5,5)(4,5.64)(5,5)(4,4.4)(3.4,5)(4,5.2)(4.4,5)(3,5)(4,5.3)(5.5,5)
(6,6)

}

\pscircle(8,1.5){1.5}
\psline[arrowsize=0.2,showpoints=false]{>->}(6.5,1.5)(9.5,1.5)
\psdot(8,1.5)

\pscircle(12,1.5){1.5}
\psdot(12,1.5)
\rput(8,-3.5){

\psecurve[arrowsize=0.2,showpoints=false]{->}(2,6)(2.5,5)(4,4)(5,5)(4,5.4)(3.4,5)(4,4.3)(4.5,5)(3.65,5)(4,4.53)(5.5,5)
(6,6)

}

 \end{pspicture}
 \caption{Skein relations for points order  $n=5$ (bottom) and the corresponding lifts (top).}\label{skein odd}
\end{figure}

\end{definition}

\begin{remark}\label{n turns} The skein relations depicted in Figures~\ref{skein even} and \ref{skein odd} imply  that a loop going  $n$ times in either direction around a point of order $n$ can be "erased" from a closed curve (Figure~\ref{skein}). However, note the the skein relation in Figure~\ref{skein} is less precise than  Definition~\ref{homotopy}. Namely, this relation  does not "tell" as   Definition~\ref{homotopy} does tell how to homotope a curve passing through a cone point. Since some geodesics do pass through cone points, we need the  skein relation in Definition~\ref{homotopy} that deals with those cases.
\end{remark}

\begin{figure}[htbp]
\begin{pspicture}(16,6.5)
%
\rput(1,1.5){Case $n=3$}
\rput(1,5){Case $n=2$}
\rput(6,5){$\sim$}
\rput(10,5){$\sim$}
\rput(6,1.5){$\sim$}
\rput(10,1.5){$\sim$}

\pscircle(4,5){1.5}
\psdot(4,5)
\psecurve[arrowsize=0.2,showpoints=false]{->}(2,4)(2.5,5)(4,5.64)(5,5)(4,4.4)(3.4,5)(4,5.2)(4.4,5)(2.5,5)(2,6)

\pscircle(8,5){1.5}
\rput(4,0){\psecurve[arrowsize=0.2]{->}(2,3)(2.5,5)(3.4,5.2)(3.4,4.8)(2.5,5)(2,6)
\psdot(4,5)}

\pscircle(12,5){1.5}
\rput(8,0){\psecurve[showpoints=false,arrowsize=0.2]{<-}(2,4)(2.5,5)(4,5.64)(5,5)(4,4.4)(3.4,5)(4,5.2)(4.4,5)(2.5,5)(2,6)
\psdot(4,5)
}

\pscircle(4,1.5){1.5}
\rput(0,-3.5){
\psdot(4,5)
\psecurve[arrowsize=0.2,showpoints=false]{->}(2,4)(2.5,5)(4,6.2)(5,5)(3.2,5)(4,5.8)(4.7,5)(3.5,5)(4,5.5)(4,4.2)(2.5,5)
(1,6)

}

\pscircle(8,1.5){1.5}
\rput(4,-3.5){\psecurve[arrowsize=0.2]{->}(2,3)(2.5,5)(3.4,5.2)(3.4,4.8)(2.5,5)(2,6)
\psdot(4,5)}

\pscircle(12,1.5){1.5}
\psdot(12,1.5)
\rput(8,-3.5){

\psecurve[arrowsize=0.2,showpoints=false]{<-}(2,4)(2.5,5)(4,6.2)(5,5)(3.2,5)(4,5.8)(4.7,5)(3.5,5)(4,5.5)(4,4.2)(2.5,5)
(1,6)
}

 \end{pspicture}
 \caption{Consequence of skein relations for points order $n=2$ (upper) and $n=3$ (lower).}\label{skein}
\end{figure}
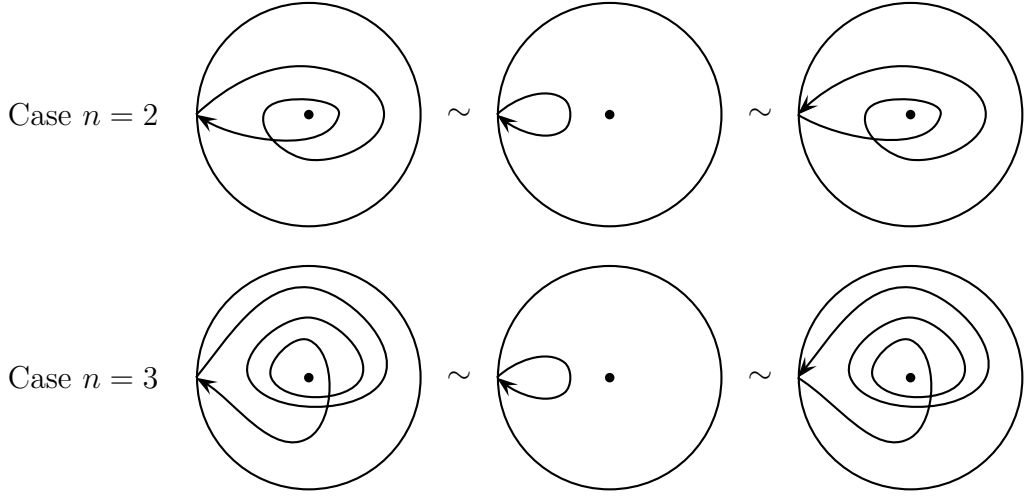

The proof of the next result is very similar to that of the (standard) proof of a bijection between
 free (usual) homotopy classes of closed curves on a path-connected space and conjugacy classes of the fundamental group of the space (see, for instance, \cite[Chapter 1, Exercise 6]{Hatcher}).

\begin{theorem}\label{bij} There is a natural bijection between the set of conjugacy classes of $G$ and the set of $\H/G$-free homotopy classes of closed oriented curves in $\H/G$.

\end{theorem}

If $a$ and $b$ in are two elements of $\H/G$,
\emph{the intersection number of $a$ and $b$} is  the minimum number (counted with multiplicity) of transversal intersection points of pairs of loops representing of $a$ and $b$ not passing through cone points. 

\begin{remark}\label{zero} If at least one of the elements,   $a$ or $b$ is  the conjugacy class of a non-hyperbolic element of $G$ then the intersection  number of $a$ and $b$ is zero. (Conjugacy classes of elements of $G$ are identified  with free homotopy classes of curves on $\H/G$ via Theorem \ref{bij}).
\end{remark}

\subsection{Labeling intersection points - The group theoretic intersection number}\label{labeling}

 A hyperbolic isometry $x$ acts on its axis $A_x$ by translation by a real number $\tax$, the  \emph{translation length of $x$}. 
We orient the axis $A_x$   so that for each point $P$ in $A_x$, the direction from $P$ to $xP$ is positive.
 
Let $x,y\in G$. Denote by $X\backslash G/Y$ the space of double cosets $XgY$ where $g \in G$,  $X$ and $Y$ denote the cyclic subgroups generated by $x$ and $y$ respectively.  If $x$ or $y$ is not hyperbolic, set $I(x,y)=\emptyset$, otherwise,  set
 $I(x,y) =\{XgY \in X\backslash G/Y \mbox{ such that  } A_x \cap hA_y \ne \emptyset \mbox{ for some $h$ in $XgY$}\}$.

(Observe that if $A_x \cap h A_y \ne \emptyset$ for some $h \in XgY$ if and only if   $A_x \cap h A_y \ne \emptyset$ for all $h \in XgY$.) 

Scott \cite{Scott} discusses intersection number of closed curves on compact surfaces. The next proposition can be proven by arguments completely analogous to those of Scott \cite[Section 1]{Scott}. The point is that $\H/G$-homotopy after lifting becomes exactly like usual homotopy in the universal cover. Thus our discussion and Scott's are "mutatis mutandi" as far as the proposition below is concerned.

\begin{proposition}\label{isom} 
Let $x$ and $y$ be elements of $G$. Then the intersection number of $\Omega x$ and $\Omega y$ equals  the cardinality of $I(x,y)$.
\end{proposition}

\section{The Goldman bracket for orbifolds}\label{orbi}
Recall that $\cc$ denotes  the set of conjugacy classes of  elements in $G$. Consider $\Z[\cc]$, the free module generated by $\cc$.
For $x\in G$,  let $\conjj{x}$ denote the conjugacy class of $x$. In particular,  $\la x\ra\in\Z[\cc]$.

In this section we will define a linear map 
$[\cdot,\cdot]:\Z[\cc]\otimes\Z[\cc]\to \Z[\cc]$  and show in Subsection \ref{Jacobi} that it is a Lie bracket. 
 This bracket generalizes Goldman's to orientable two dimensional orbifolds and will be defined (as Goldman's) on two elements of the basis of  $\Z[\cc]$ by considering the intersection points of certain pair of representatives (Subsection~\ref{labeling}) , assigning a signed free homotopy class to each of these points  (the signed product at the intersection point) and adding up all those terms.

 For elements $a$ and $x$  in $ G$,  let $x^a$ denote $axa^{-1}$.  The isometry $x^a$ is also a hyperbolic, it has the same   translation length as $x$,  $\taxx{x^a}=\tax$, and the axis of $x^a$ given by $a\cdot A_x$. From now on, fix an orientation of $\H$. Also, for $x$ and $y$ in $G$ set  $\inu(x,y)$ to be zero  if the axes of  $x$ and $y$ do not cross and to be   the sign of the crossing, otherwise. Finally, set
\begin{equation}\label{brackdblcst}
[\conjj{x},\conjj{y}]=\sum_{XbY\in I(x,y)} \inu(x,y^b)\la xy^b\ra. 
\end{equation}

\begin{notation}\label{segment} Let  $P $ be a point in the axis $A_x$ of a hyperbolic transformation $x$. If $r$ is a positive real number,  $S(x,P,r)$ denotes  the segment of  $A_x$ of length $r$ starting (and including)   $P$,  but not the other endpoint, in the positive direction of $A_x$.   If $r$ is a negative number, $S(x,P,r)$  denotes the segment starting at a point $Q$ at distance $r$ of $P$ in the negative direction, containing $Q$ but not $P$. 
\end{notation}

\begin{remark}\label{pairs1} Fix a point $P$ in $ A_x$ and let $r$ be the translation length of $x$. Let $$J(x,y,P)=\{gY \in G/Y \mbox{ such that  } S(x,P,r)\cap hA_y \ne \emptyset \mbox{ for some } h \in gY\}.$$  Then there is a bijection between $I(x,y)$   and $J(x,y,P)$. Since $G$ is a discrete group, both sets have finite cardinality. Moreover,
\begin{equation}\label{brackdblcst1}
[\la x\ra,\la y \ra]=\sum_{gY\in J(x,y,P)} \inu(x,y^g)\la xy^g\ra. 
\end{equation}

 \end{remark}

\begin{remark}\label{2} The conjugacy classes of elliptic and parabolic elements of $G$ are in the center of the Lie algebra, that is, the bracket between these classes and all other classes is zero.
 \end{remark}

\begin{remark}\label{product} By \cite[Theorem 7.38.6]{B}, if $x$ and $y$ are hyperbolic isometries whose axes intersect then $xy$ is also hyperbolic. Moreover, the axis of $xy$ and its translation length can be determined as follows (see \cite{B} for details). Denote by $P$ the intersection point of $A_x$ and $A_y$.  
Denote by $Q$ the point on $A_x$ at distance $\tax/2$ of $P$ in the positive direction of $A_x$ and by $R$ the point on $A_y$ at distance $\tay/2$ of $P$ in the negative direction of $A_y$. The axis of $A_{xy}$ is the oriented line from $R$ to $Q$ and the translation length of $xy$ equals twice the distance between $R$ and $Q$. (See Figure \ref{m3}. This is one of the "triangles" mentioned in the introduction which are used to unravel the Jacobi relation.)

\begin{figure}[htbp]
\begin{pspicture}(16,3.4)

\rput(7.12, 2){$\tax/2$}
\rput(8.88, 2){$\tay/2$}
\rput(8, 0.68){$\taxy /2$}

\rput(6,0){
\psline{<-}(2.5,3)(0,0) 
\psline{->}(1.5,3)(4,0)
\psline{->}(0,1)(4,1)

\psdot(2,2.42)
\rput(2,2.8){$P$}

\psdot(3.2,1)
\rput(3.1,0.7){$Q$}

\psdot(0.8,1)
\rput(0.8,0.7){$R$}

\rput(1.2,3.1){$A_x$}
\rput(2.8,3.){$A_y$}

\rput(4.3,1){$A_{xy}$}

}

 \end{pspicture}
 \caption{The axis of $xy$}\label{m3}
\end{figure}
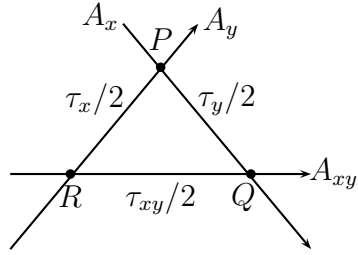

\end{remark}

\begin{remark}\label{pairs} Consider the set of  pairs of cosets $G/X \times G/Y$. The group $G$ acts on the set $G/X \times G/Y$ by $(Xg, Yh)\mapsto (Xga, Yha)$, for each $a \in G$. Denote by $D(x,y)$ the quotient under this action. Set $\map{f}{D(x,y)}{X\backslash G/Y}$ by mapping the equivalent class of $(Xg, Yh)$ to
$Xgh^{-1}Y$. A straightforward computation shows that $f$ is well defined and it is a bijection.  Also, the preimage under $f$ of an element $XkY$ of $I(x,y)$ is the set of equivalence classes of pairs of cosets $(Xg, Yh)$ such that $gA_x \cap hA_h \ne \emptyset$ and $gh^{-1}=k$. Moreover,
\begin{equation}\label{bracket}
[\la x\ra,\la y \ra]=\sum_{(Xa, Yb)\in D(x,y)} \inu(x^a,y^b)\la x^ay^b\ra. 
\end{equation}

 \end{remark}

\section{Triple brackets and the Jacobi identity}\label{Jacobi}

Jacobi identity for the extended bracket can be probably proved by argument analogous to those used by Goldman in his proof that the bracket of curves on surfaces statisfies the Jacobi identity.

In this section we present a geometric proof of the Jacobi identity, that does not use transversality.

Let $x$ be  a hyperbolic isometry and let $P \in A_x$. The next result is stated using Notation \ref{segment}

\begin{lemma}\label{triple} The following equation holds.
$$[[\conjj{x},\conjj{y}],\conjj{z}]=\sum_{(XgY,XhZ)\in T} \inu (x,y^g)\inu(x,z^h)\la xy^gz^h\ra +
\sum_{(XgY,YhZ)\in U} \inu (x,y^g)\inu(y^g,z^{h})\la xy^gz^h\ra.$$
where $$T=\{(XgY,XhZ) :  A_x \cap gA_y = \{P\}, S(x,P,\tax) \cap h A_z \ne \emptyset ,$$$$ h A_z \cap (S(y^g,P, -\tay/2)\cup S(xP,y^{xg},\tay/2)=\emptyset \}\mbox{ and }$$  
$$U=\{(XgY,YhZ) :  A_x \cap gA_y = \{P\}, (S(P,y^g,-\tay/2)\cup S(xP,y^{xg},\tax/2))\cap h A_z \ne \emptyset ,$$$$  S(P,x,\tax)  \cap h A_z = \emptyset\}.$$
\end{lemma}
\begin{proof}

Let $g \in G$ such that $A_x \cap g A_y \ne \emptyset$.  We can retrace the steps of the construction described in Remark \ref{product} to find $A_{xy^g}$ (Figure \ref{m4}).  Next, we  compute $[\la xy^g\ra,\la z\ra]$.  Denote by $P$ the intersection point between $A_x$ and $gA_y$, by $S$ the intersection point of $A_x$ with $A_{xy^g}$ and by  $R$  the  intersection point of $gA_y$ and $A_{xy^g}$. Finally, denote by $Z$ the cyclic group generated by $z$. By Remark \ref{pairs1}
$$
[\la xy^g\ra,\la z\ra]=\sum_{\substack{hZ \in G/Z, \\
S(R,xy^g, \taxy ) \cap h A_z \ne \emptyset}} \inu(xy^g, z^h) \la xy^gz^h\ra.
$$
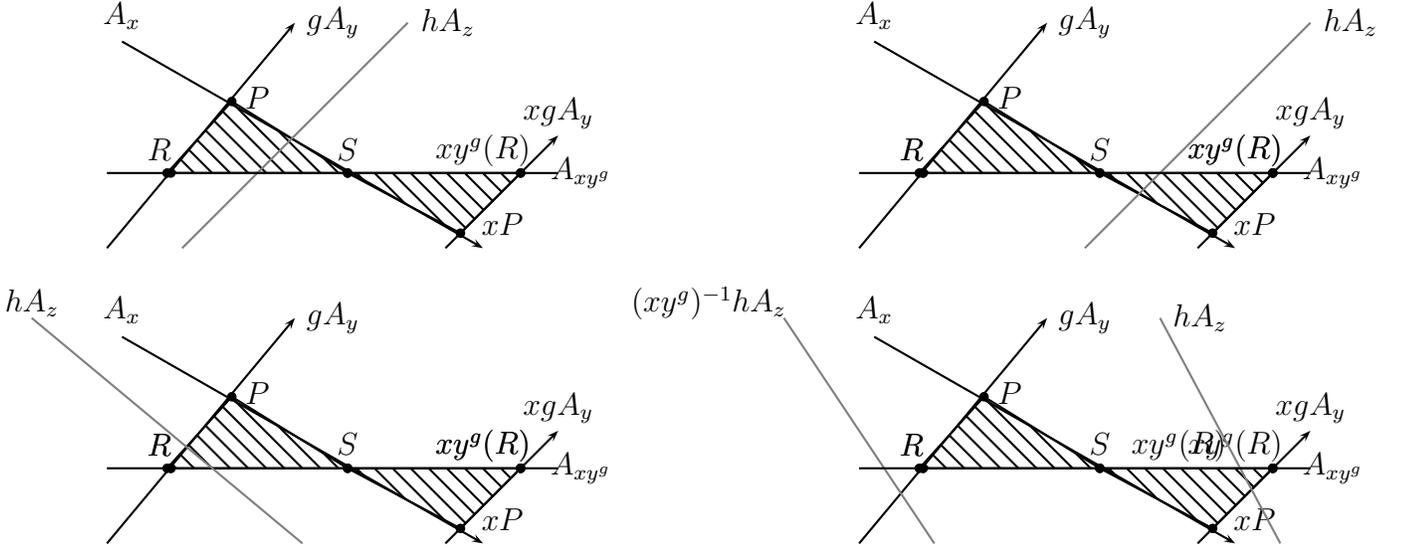
\begin{figure}[htbp]
\begin{pspicture}(16,3.4)

\psdot(0.85,1)\rput(0.7,1.31){$R$}
\psdot(5.5,1)\rput(5,1.3){$xy^g(R)$}
\pspolygon[showpoints=true,fillstyle=vlines](0.8,1)(1.66,1.952)(3.2,1)
\pspolygon[showpoints=true,fillstyle=vlines](5.5,1)(4.7,0.2)(3.2,1)
\rput(2,2){$P$}
\rput(5.26,0.322){$xP$}
\rput(3.2,1.32){$S$}

\rput(0.2,3.1){$A_x$}\psline{->}(0.2,2.75)(5,0.)

\rput(3,3.){$gA_y$}\psline{<-}(2.5,3)(0,0)
\rput(6.,1.8){$xgA_y$}\psline[]{<-}(6,1.5)(4.5,0)

\rput(4.53,3.){$hA_z$} \psline[linecolor=gray](4,3)(1,0)

\rput(6.3,1){$A_{xy^g}$}\psline(0,1)(6,1)

\rput(10,0){
\psdot(0.85,1)\rput(0.7,1.31){$R$}
\psdot(5.5,1)\rput(5,1.3){$xy^g(R)$}
\pspolygon[showpoints=true,fillstyle=vlines](0.8,1)(1.66,1.952)(3.2,1)
\pspolygon[showpoints=true,fillstyle=vlines](5.5,1)(4.7,0.2)(3.2,1)
\rput(2,2){$P$}
\rput(5.26,0.322){$xP$}
\rput(3.2,1.32){$S$}

\psdot(0.85,1)\rput(0.7,1.31){$R$}

\psdot(5.5,1)\rput(5.,1.3){$xy^g(R)$}

\rput(0.2,3.1){$A_x$}\psline{->}(0.2,2.75)(5,0.)

\rput(3,3.){$gA_y$}\psline{<-}(2.5,3)(0,0)

\rput(6.,1.8){$xgA_y$}\psline[]{<-}(6,1.5)(4.5,0)

\rput(6.53,3.){$hA_z$} \psline[linecolor=gray](6,3)(3,0)

\rput(6.3,1){$A_{xy^g}$}\psline(0,1)(6,1)

}

 \end{pspicture}

\hspace{1cm}

\begin{pspicture}(16,3.4)
\psdot(0.85,1)\rput(0.7,1.31){$R$}
\psdot(5.5,1)\rput(5,1.3){$xy^g(R)$}
\pspolygon[showpoints=true,fillstyle=vlines](0.8,1)(1.66,1.952)(3.2,1)
\pspolygon[showpoints=true,fillstyle=vlines](5.5,1)(4.7,0.2)(3.2,1)
\rput(2,2){$P$}
\rput(5.26,0.322){$xP$}
\rput(3.2,1.32){$S$}

\psdot(0.85,1)\rput(0.7,1.31){$R$}

\psdot(5.5,1)\rput(5,1.3){$xy^g(R)$}

\rput(0.2,3.1){$A_x$}\psline{->}(0.2,2.75)(5,0.)

\rput(3,3.){$gA_y$}\psline{<-}(2.5,3)(0,0)

\rput(6.,1.8){$xgA_y$}\psline[]{<-}(6,1.5)(4.5,0)

\rput(6.3,1){$A_{xy^g}$}\psline(0,1)(6,1)

\rput(-1,3.2){$hA_z$} \psline[linecolor=gray](-1,3)(2.6,0)

\rput(10,0){

\psdot(0.85,1)\rput(0.7,1.31){$R$}
\psdot(5.5,1)\rput(5,1.3){$xy^g(R)$}
\pspolygon[showpoints=true,fillstyle=vlines](0.8,1)(1.66,1.952)(3.2,1)
\pspolygon[showpoints=true,fillstyle=vlines](5.5,1)(4.7,0.2)(3.2,1)
\rput(2,2){$P$}
\rput(5.26,0.322){$xP$}
\rput(3.2,1.32){$S$}

\psdot(0.85,1)\rput(0.7,1.31){$R$}

\psdot(5.5,1)\rput(4.25,1.3){$xy^g(R)$}

\rput(0.2,3.1){$A_x$}\psline{->}(0.2,2.75)(5,0.)

\rput(3,3.){$gA_y$}\psline{<-}(2.5,3)(0,0)

\rput(6.,1.8){$xgA_y$}\psline[]{<-}(6,1.5)(4.5,0)

\rput(4.53,3.){$hA_z$} \psline[linecolor=gray](4,3)(5.6,0)

\rput(-2,3.2){$(xy^g)^{-1}hA_z$} \psline[linecolor=gray](-1,3)(1,0)

\rput(6.3,1){$A_{xy^g}$}\psline(0,1)(6,1)

}
\end{pspicture}

\caption{Jacobi Identity}\label{m4}
\end{figure}
Let $hZ \in G/Z$. Observe that the inequality $ I_{xy^g}^P \cap h A_z \ne \emptyset$ holds if and only if $hA_z$ crosses either the triangle with vertices   $R, P, S$ or the triangle with vertices $S, xP, xy^gR$  (Figure \ref{m4}).Thus, $hA_z$ intersects $I_{xy^g}^P$ if and only if   exactly one of the following holds:
\begin{numlist}
\item $S(P,x,\tax) \cap h A_z \ne \emptyset$ and $(S(y^g,P, -\tay/2)\cup S(xP,y^{xg},\tay/2))\cap h A_z = \emptyset$, or
\item $S(P,x,\tax) \cap h A_z =\emptyset$ and $(S(y^g,P, -\tay/2)\cup S(xP,y^{xg},\tay/2) \cap h A_z \ne \emptyset$. 
\end{numlist}
The first two condition corresponds to a term in the first  sum and the second condition, to terms in the second sum. Thus, this concludes the proof.

\end{proof}
 
A  corollary is the Jacobi identity.
\begin{theorem}
For $x,y,z\in G$,
$$[[x,y],z]+[[y,z],x]+[[z,x],y]=0.$$
Therefore, $[\cdot,\cdot]:\Z[\cc]\otimes\Z[\cc]\to \Z[\cc]$ is a Lie bracket.
\end{theorem}
\begin{proof} The three terms of the Jacobi relation after applying  Lemma~\ref{triple} decompose into in six groups of  terms.  Among these, the pairs corresponding to the triangles of Figure~\ref{m4} cancel.
\end{proof}

\section{Examples}\label{examples}
Consider the modular group $PSL(2, \Z)$, that is,  the group consisting of all transformations $z \longrightarrow (a z +b)/(c  z+d)$, where $a,b,c,d \in \Z$ and $a d - b c =1$. This group is generated by $T(z) =z+1$ and $S(z) = -1/z$, with relations $S^2=1$ and $(ST)^3=1$.  The modular group is a finitely generated, discrete subgroup of orientation preserving isometries of the hyperbolic plane. Therefore, the bracket can be defined on the free module generated by conjugacy classes.  

Orient the hyperbolic plane clockwise.

By computing the traces,  one can see that the elements $x=TSTT$ and $y=TTTSTTT$ of $PSL(2,\Z)$ are hyperbolic and not conjugate.

\begin{figure}[htbp]
\begin{center}

\includegraphics[width=10cm,angle=90]{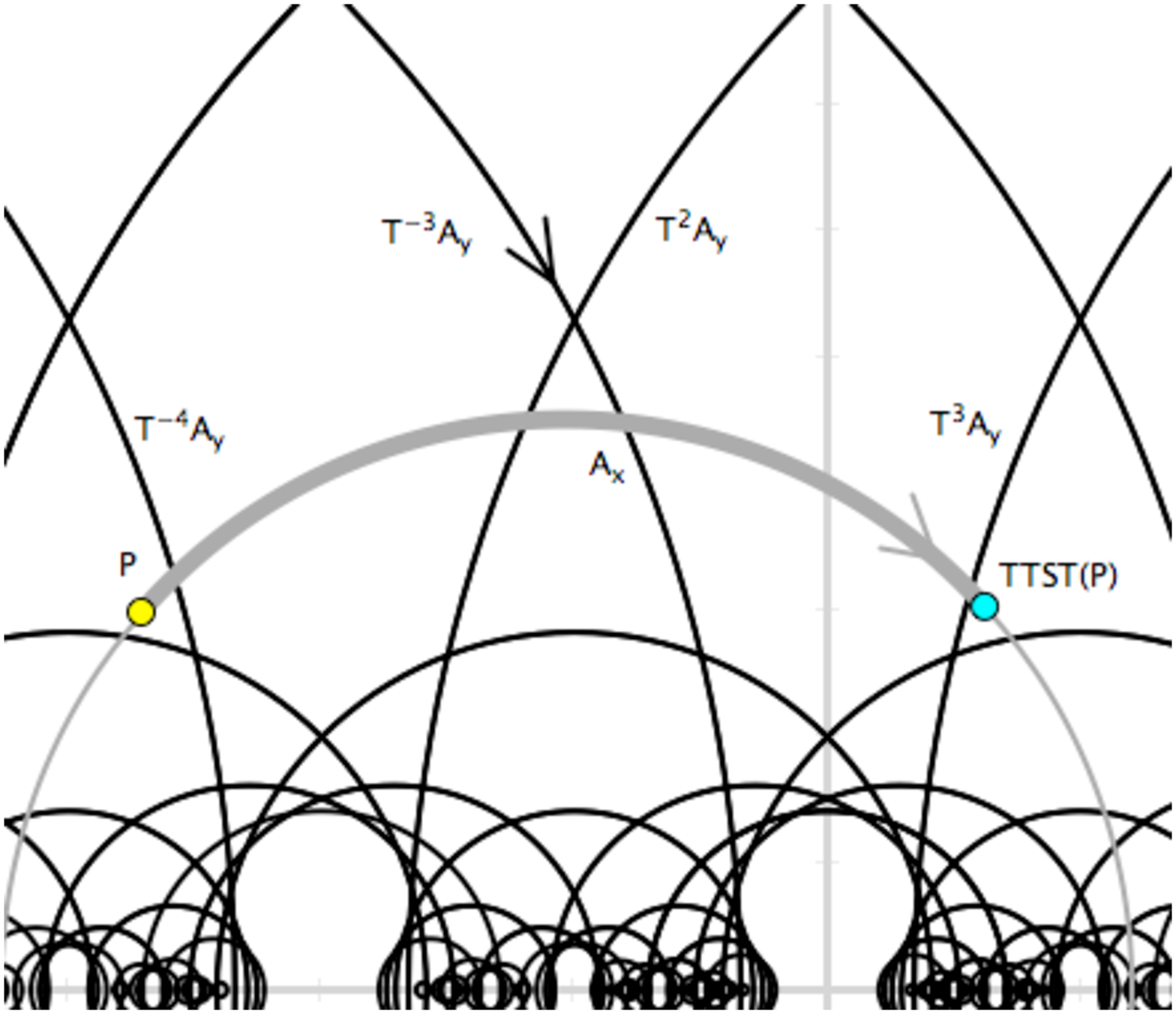}

\includegraphics[width=10cm,angle=90]{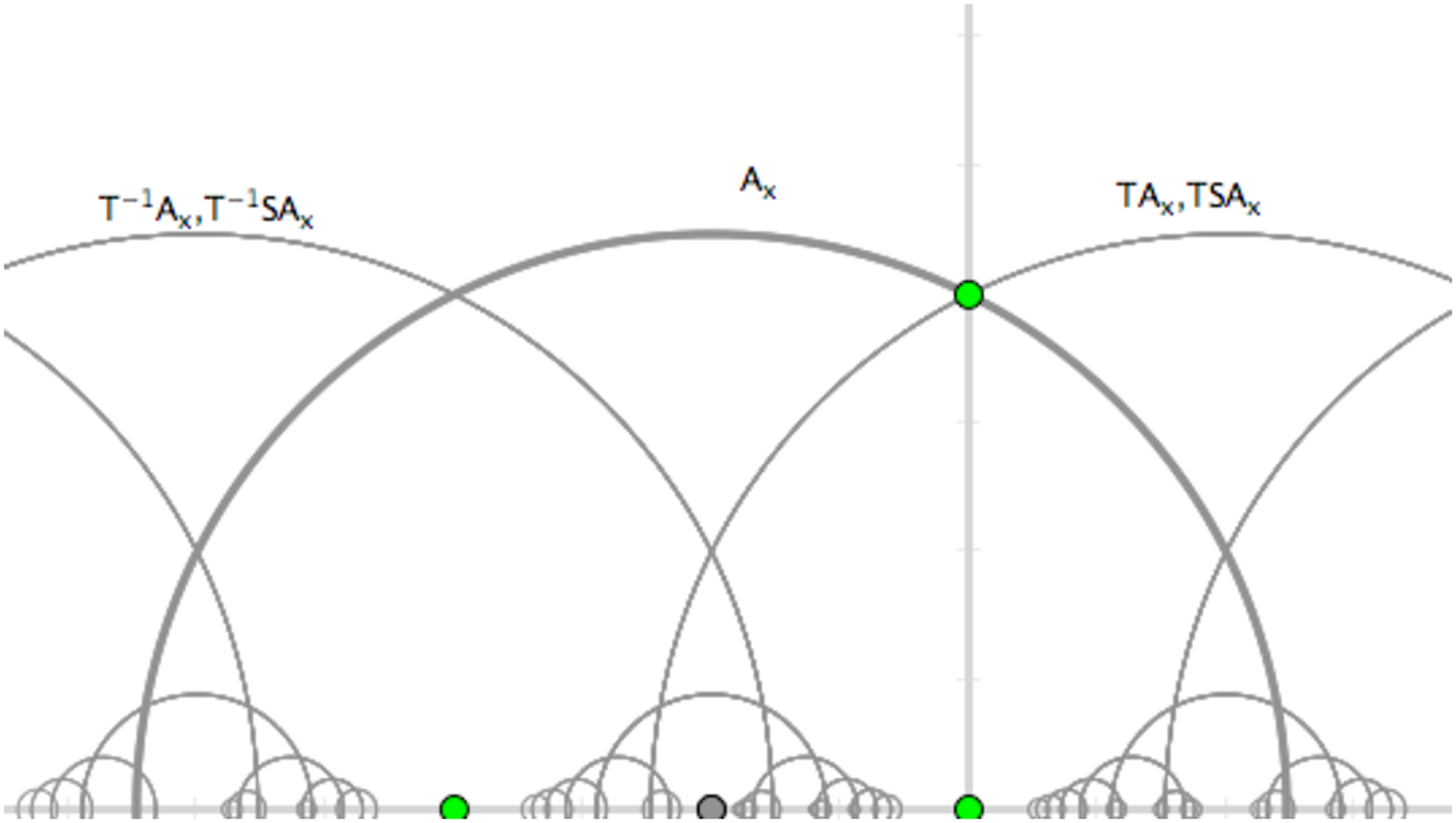}
\end{center}
\caption{Translates of $A_y$ (in black), and a fundamental domain of $A_x$  (in thick gray) where $x=TSTT$ and $y=TTTSTTT$ }\label{example}
\end{figure}

 As shown in Figure \ref{example}, there are exactly four translates of $y$ by $PSL(2,\Z)$ that intersect the segment of $A_x$  from  the point $P$  to $TTST(P)$.

In this example, $I(x,y)=\{X T^{-4} Y, X T^{-3} Y ,X T^{2} Y,X T^{3} Y\} $. A direct computation shows that  terms associated to the double cosets   $X T^{-4} Y$ and $XT^3Y$ are respectively $+\conjj{ST^6}$ and $-\conjj{ST^6}$. Also, the terms associated to $XT^{-3}$ and $XT^2$ are $+\conjj{STTST^7}$ and $-\conjj{STTST^7}$. Thus $[\conjj{x},\conjj{y}]=0$. 

In order to study  the brackets of  $\conjj{x^p}$ and $\conjj{y^q}$ when $p$ and $q$ are larger than one, one can use the criteria given in \cite{Katok1} for conjugacy in $SL(2,\Z)$ (and therefore in $PSL(2,\Z)$). Doing so, one can check that $[\conjj{x},\conjj{y^3}] \ne 0$. Moreover, the number of terms of  the bracket $[\conjj{x},\conjj{y^3}]$ (counted with multiplicity) equals twelve, which is three times the intersection number of $\conjj{x}$ and $\conjj{y}$.   

In the same way one can see that the  $[\conjj{x},\conjj{x^2}] = 0$ and $[\conjj{x},\conjj{x^3}] $ has 24 terms which is six times the self-intersection number of $\conjj{x}$.

The above calculations are computer assisted: One looks at Figure~\ref{example} (done with Cinderella) to identify the terms. Then uses Mathematica to calculate the terms, and study cancellation.

\section{Quantitative separation  of geodesics}\label{dog}

From now on, we assume that the discrete subgroup $G$ of $\mathrm{Isom}(\H)$  is finitely generated.

\begin{definition}\label{delta}
Fix  $\clo>0$, two geodesics  $\geo$ and $\geop$ and two (not necessarily distinct) points $P$ and $Q$ in $\geo$ and $\geop$ respectively. We say that $\geo$ and $\geop$ are \emph{$\clo$-close at $P$ and $Q$} if $d(P,Q)< \clo$ and,   if $\Upsilon$ denotes a geodesic passing through $P$ and $Q$, then the absolute value of the difference between the corresponding  angles between $\Upsilon$ and $A_x$ and $\Upsilon$ and $A_y$ (in the positive direction of both axes) is less than $\delta$.   If there exist points $P$ and $Q$ such two geodesics $\geo$ and $\geop$ are $\delta$-close at $P$ and $Q$, then we say that \emph{$\geo$ and $\geop$ are $\delta$-close}.
\end{definition}

The next lemma is well known to experts but we include a proof here  because we were unable to find one in the literature. 

\begin{lemma}\label{diverg} For each $L >0$ there exists $\clo>0$ such that if  $x$ and $y$ are two hyperbolic transformations in $G$  such that  $\tax \le L$ and $\tay \le L$ and  $A_x$ and $A_y$ are  $\clo$-close then $A_x= A_y$. 
\end{lemma}
\begin{proof} 

Denote by $\Lambda$ the hyperbolic convex hull of the limit set of $G$. (Recall  the limit set of $G$ is the set of accumulation points of any $G$-orbit in $\H$.) 
Since $G$ is finitely generated, by  \cite[Lemma  1.3.1 and Theorem 1.3.2]{Green}, there exist a subset $\Lambda^*$ of $\Lambda$, invariant under $G$ such that the quotient of $\Lambda^*$ by $G$ is compact and the axis of every hyperbolic transformation in $G$  intersects $\Lambda^*$. Thus,  there exist a compact, convex  subset $C$ of $\H$ such that $\Lambda^* \subset G.C$.

Fix a positive number $L$ and denote by $C'$ the closure of the $(L+1)$-neighborhood of $C$.

\textbf{Claim 1:} Given $\varepsilon>0$ there exist $\delta>0$ such that if $x$ and $y$ are hyperbolic transformations whose axes are $\delta$-close and whose transformation lengths are bounded above by $L$ then $d(R, [x,y]R)<\varepsilon$ for all $R \in C'$.

We argue by contradiction: Suppose that there exist $\varepsilon>0$ and  two sequences $\{x_n\}$ and $\{y_n\}$ of hyperbolic transformations with translation length bounded above by $L$ and such that for each $n$, $x_n$ and $y_n$ are $1/n$-close,  $A_{x_n} \ne A_{y_n}$  and there exists a point $R_n \in C'$ that satisfies $d(R_n,[x_n,y_n]R_n)>\varepsilon$.

 \textbf{Claim 2:} For each $n$, we can assume that the points $P_n$ and $Q_n$  in $A_{x_n}$ and $A_{y_n}$ realizing Definition~\ref{delta} are in $C'$.

 Indeed, denote by $P_n'$ and $Q_n'$  the points in $A_{x_n}$ and $A_{y_n}$ realizing  Definition~\ref{delta}.

The axis $A_{x_n}$ projects to a closed geodesic $a_n$ in  $\H/G$. Since the translation length of $x_n$ is bounded by above by $L$, so  is the length of $a_n$. On the other hand,   $A_{x_n}$ intersects $G.C$. Hence, the projection of $P_n'$ to $\H/G$ is at distance at most $L$ from the projection of $G.C$.  Thus there is an element $g \in G$ such that $gP_n'$ is at distance at most $L$ of $C$. Since $Q_n'$ is close to $P_n'$, $Q_n'$ is also in $C'$. The proof of Claim 2 is completed by replacing the sequences $\{x_n\}$ and $\{y_n\}$ by the sequences $\{gx_ng^{-1}\}$ and $\{gy_ng^{-1}\}$. 

\textbf{Claim 3:} The sequences $\{x_n\}$ and $\{y_n\}$ have subsequences converging to hyperbolic transformations $x$ and $y$ respectively. 

Consider the sequences $\{T_n\}$ and $\{S_n\}$ of endpoints of $\{A_{x_n}\}$  in the circle at infinity in the negative and positive directions respectively. Since the circle is compact, by taking subsequences, we can assume that  $\{T_n\}$ and $\{S_n\}$ converge to $T$ and $S$ respectively. Since each $A_{x_n}$ intersects the compact set $C'$, $T \ne S$. Analogously, the sequence  $\{\taxx{x_n}\}$ of translation lengths is bounded by above by $L$. Therefore, it has a convergent subsequence. Thus,  Claim  3 follows. 

Since $A_{x_n}$ and $A_{y_n}$ are $1/n$-close, $A_x=A_y$.  Hence, $[x,y]P=P$ for all $P \in \H$. On the other hand, by taking a convergent subsequence of $\{R_n\}$, we see that $d(R,[x,y]R)\ge \varepsilon$
for some $R \in C'$.  This contradiction completes the proof of Claim 1.

To finish  the proof of the lemma, observe that since $G$ is discrete, there exists an open subset $U$ of isometries of $\H$ such that the identity is the only element of $G$ in $U$. 
Let
$$V_\eta
=\{ g \in PSL(2,\R) , d(R,gR)<\eta \mbox{ for all $R$ in $C'$}\}.$$
There exists $\varepsilon>0$ such that $V_\varepsilon \subset U$.   On the other hand, by Claim 1, there exists $\delta>0$ such that if the axes of $x$ and $y$ are $\delta$-close, then $[x,y] \in V_\varepsilon$.  Thus, the bracket $[x,y]$ equals the identity, which implies $A_x=A_y$.
\end{proof}

\begin{corollary}\label{smallint} For each $L>0$ and each $C>0$ there exists a constant $M>0$ such that for every pair of hyperbolic elements  
 $x$ and $y$ in $G$ with different axes and  such that $\tax <L$ and $\tay<L$,  the set 
$A_x \cap N_C(A_y)$
is a (possibly empty) geodesic segment of length at most $M$.
\end{corollary}
\begin{proof} 
Let $\clo$ be is as in  Lemma \ref{diverg} for $L$ and $G$ and let $N$ be the length of the (possibly empty) segment $A_x \cap N_C(A_y)$. 

If $A_x$ and $A_y$ intersect at an angle $\theta$, then by Lemma \ref{diverg}, $\sin(\theta) \ge \sin(\delta)$.
 By the  Rule of Sines, $\sinh(N/2)\le   \sinh(C)/\sin(\clo)$ (see Figure \ref{m}(a).) Then $N$ is bounded  above by a constant depending on $C$ and $\delta$. 

If $A_x$ and $A_y$ are parallel, by Lemma \ref{diverg} they are at distance at least $\clo$. Since the distance between $A_x$ and $A_y$  is realized, there is a quadrilateral  as in Figure \ref{m}(b), with all angles except $\theta$ being right angles, $A \ge \clo$ and $B \le C$.

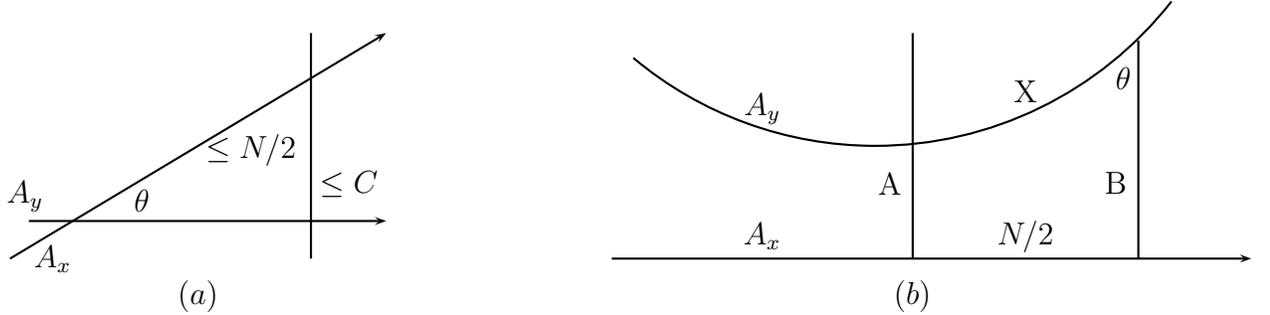
\begin{figure}[htbp]
\begin{center}
\begin{pspicture}(14,3)%

\psline{->}(0.25,0.5)(5,0.5)\rput(0.2,0.83){$A_y$}
\psline{->}(0,0)(5,3)\rput(0.57,0){$A_x$}

\psline(4,0)(4,3)
\rput(3.2,1.45){$\le N/2$}
\rput(4.5,1){$\le C$}
\rput(1.7539,0.775){$\theta$}

\rput(2.5,-0.5){$(a)$}

\rput(8,0){

\psline{->}(0,0)(8.5,0)\rput(2,0.3){$A_x$}

\psline(4,0)(4,3) \rput(3.7,1){A}
\psline(7,0)(7,2.9) \rput(6.7,1.012){B}
\psarc(3.5,6.5){5}{230}{322}\rput(2,2){$A_y$}
\rput(5.5, 0.3){$N/2$}
\rput(6.8,2.4){$\theta$}
\rput(5.5,2.23){X}
}

\rput(12,-0.5){$(b)$}
 \end{pspicture}

\end{center}
\caption{Proof of Lemma \ref{smallint}}\label{m}
\end{figure}
By \cite[Theorem 7.17.1(i)]{B},  $\sinh(N/2)=\cos(\theta)/\sinh(A) \le 1/\sinh(\clo)$ (see Figure~\ref{m}(b)). This implies that $\cosh(N/2)$ is bounded by above by a bound  depending on $\clo$. An elementary computation gives the desired result.
\end{proof}

\section{The Non-cancellation Lemma}\label{ncl}

Let $K$ be a real positive number. A piecewise-smooth embedding $\tg$ of $\R$ in the hyperbolic plane is a \emph{$K$-quasi-geodesic} if for any pair of points $P$ and $Q$ in $\tg$, the length of the path in $\tg$ joining $P$ and $Q$ is at most $K \cdot d(P,Q)$.

Fix a pair of hyperbolic elements $x$ and $y$ in $G$ whose edges intersect at a point $P$.  We will describe the construction of a piecewise-smooth embedding $\tg$ of $\R$ (depending on $x$ and $y$) and   show it is a quasi-geodesic. 

Let $\map{\alpha}{[0,1]}{\H}$ be the curve  from $\alpha(0)=y^{-1}P$ to $\alpha(1)=x P$, whose image is given by
the concatenation of the geodesic segment  of $A_y$ from  $y^{-1}P$ to $P$ with the geodesic segment of $A_x$ from $P$ to $x P$. 
 Since $x y(\alpha(0))=\alpha(1)$,  $\alpha$ can be extended by periodicity to a map $\tg(x,y):\R\to\H\thinspace$ such that   $\tg(x,y)(t)=\alpha(t)$ for $t\in [0,1]$ and $\tg(x,y)(t+1)=x y\tg(x,y)(t)$ for all $t$.

The map $\tg(x,y)$ is a piecewise geodesic curve consisting of segments of length $\tax$ (included in the axes of  conjugates of $A_x$ by some power  of $xy$) alternating with segments of length $\tay$ (included in the axes of  conjugates of $A_y$ by some power  of $xy$.)

We remark that we will be using more than just that $\tg(x,y)$ is a quasi-geodesic, but also its geometric nature. Indeed purely abstract results about quasi-geodesics suffice to prove a weaker version of our result, where we need to assume that both $p$ and $q$ are large.

\begin{lemma}\label{pq} For each $L>0$ 
 there exists a constant $K>0$ depending on $G$ such that  if $x$ and $y$ are hyperbolic transformations in $G$ whose axes are distinct and intersect, and whose translation lengths are bounded   above by $L$  then for each pair of positive integers $p$ and $q$, the curve $\tg(x^p,y^q)$ is a $K$-quasi-geodesic. Moreover, the oriented angles between any pair of consecutive maximal segments of $\tg(x^p,y^q)$ are congruent.
\end{lemma}
\begin{proof} 
Fix $p$ and $q$ and repeat the construction of Remark \ref{product} for the hyperbolic isometries $x^p$ and $y^q$.
The transformation $x^p$ maps the angle determined by $y^{-q}P, P, x^p(P)$ to the angle $x^py^{-q}P, x^pP, x^{2p}(P)$ (Figure \ref{qgeo1}). Thus, these two angles are congruent. The angle $x^py^{-q}P, x^pP, x^{2p}(P)$ is congruent to  the angle $P, x^{p}P,x^{p}y^{q}(P)$ because they are opposite at the intersection of $A_{x}$ and $x^{p}y^{q}(A_{y})=A_{x^{p}y x^{-p}}$. This implies that the angles determined by $y^{-q}P, P, x^p(P)$ and by $P, x^p(P), y^q x^p(P)$ are congruent.  Therefore the angles formed by the consecutive maximal segments of  $\tg(x^p,y^g) $  (labeled with $\theta_{1}$ in Figure \ref{qgeo1}) are all congruent.

\begin{figure}[htbp]
\begin{center}
\begin{pspicture}(0,1)(12,6)%

\rput(11.5,5){$\Omega$}
\psline[linestyle=dotted,arrowsize=0.2]{->}(11.5,4.975)(6,4.5)
\psline[linestyle=dotted,arrowsize=0.2]{->}(11.5,4.975)(6.6,3.5)
\psline[linestyle=dotted,arrowsize=0.2]{->}(11.5,4.975)(4.6,3.5)
\psline[linestyle=dotted,arrowsize=0.2]{->}(11.5,4.975)(3,2)
\psline[linestyle=dotted,arrowsize=0.2]{->}(11.5,4.975)(6,2)
\pspolygon[fillstyle=vlines,hatchwidth=0.01pt, fillcolor=lightgray,linecolor=white](2,1)(6,1)(3,3)
\pspolygon[fillstyle=hlines,hatchwidth=0.1pt,linecolor=white](3,3)(6,1)(7,3)
\pspolygon[fillstyle=vlines,hatchwidth=0.1pt, fillcolor=lightgray,linecolor=white](3,3)(7,3)(4,5)
\pspolygon[fillstyle=hlines,hatchwidth=0.1pt,linecolor=white](4,5)(8,5)(7,3)

\psline[linestyle=dashed]{->}(1,1)(10,1) \rput(9,1.2){$A_y$}
\psline[linestyle=dashed]{->}(7,0.3)(0.1,5)\rput(7,0.6){$A_x$} 

\psline[linestyle=dashed](0,3)(10,3)\rput(10,3.2){$A_{x^{p}y^q x^{-p}}$}

\psline[linewidth=0.05]{>-}(2,1)(6,1)
\psline[linewidth=0.05](6,1)(3,3)

\psline[linewidth=0.05](3,3)(7,3)
\psline[linewidth=0.05](7,3)(4,5)

\psline[linewidth=0.05](4,5)(8,5)

\psdot(0.3,3) \rput(0.6,3.4){$x^py^{-q}(P)$}

\psdot(0.3,4.9)\rput(1,5){$x^{2p}(P)$}


\rput(5,1.3){$\theta_1$}
\rput(4,2.7){$\theta_1$}

\rput(6.1,3.3){$\theta_1$}
\rput(2.3,3.25){$\theta_1$}


\rput(2.5,1.3){$\theta_2$}
\rput(3.5,3.3){$\theta_2$}

\rput(6.52,2.6){$\theta_2$}

\rput(7.52,4.6){$\theta_2$}

\rput(3.1,2.5){$\theta_3$}
\rput(4.1,4.5){$\theta_3$}

\rput(5.85,1.4){$\theta_3$}

\rput(1.5,0.7){$y^{-q}(P)$}
\rput(5.9,0.7){$P$}
\rput(2.4,2.7){$x^p(P)$}
\rput(8,3.2){$g(P)$}
\rput(3.4,5.4){$x^py^{q}x^p(P)$}

\rput(9,5.){$(g)^2(P)$}

\rput(5.9,5.63){$A_{g}$}
\psline[linecolor=gray,linewidth=2pt](3.8,0.6)(6.6,5.8)

\psdot(4,1) \rput(4,0.5){$R$}
\psdot(4.5,2) \rput(5,2.2){$Q$}

 \end{pspicture}

\end{center}
\caption{Quasigeodesic associated to $x$, $y$, $p$ and $q$}\label{qgeo1}
\end{figure}
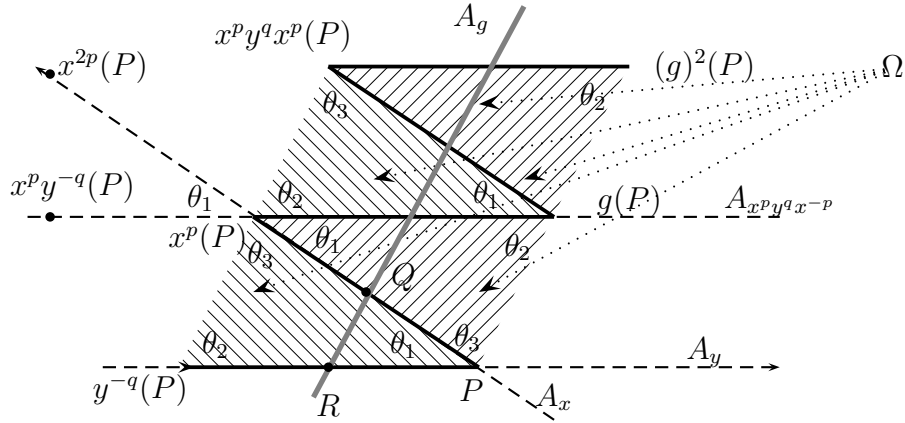

Denote by  $T$ the triangle with vertices, $y^{-q}P, P, x^p(P)$ and by $T'$ the triangle vertices $P, x^p(P),g(P)$ (Figure \ref{qgeo1}). Since $T$ and $T'$ have an angle and the two adjacent sides to the angle congruent, they are congruent.

Set $g=x^py^q$. Since $A_g$ is invariant under $g$, $A_g$ crosses the middle of  the band $\cup_{k \in \Z} g^{k}(T \cup T')$.

To prove that  $\tg(x^p,y^q)$ is a quasi-geodesic, observe  that triangles 
$$g^s(T), g^s(T'),g^{s+1}(T), g^{s+1}(T'),\dots , g(T), g(T'),\dots,g^{r}(T), g^{r}(T')$$ form a polygon $\Omega$. On the other hand, since the angles   $\theta_1, \theta_2$ and $\theta_3$ (see Figure \ref{qgeo1}) are the interior angles of a triangle, they add up to at most $\pi$. This implies that the polygon $\Omega$ is convex. 
Therefore, the geodesic between two points in the curve $\tg$ is in the interior of  $\Omega$. By elementary hyperbolic geometry, there exists a positive constant $K$ such that $\tg$ is a $K$-quasi-geodesic, (Note that $K$ can be taken so that it depends only on the lower bound of the angle between intersecting elements of axes of hyperbolic elements in $G$ given by Lemma \ref{diverg}.)

\end{proof}

\begin{lemma}\label{pq1} Let $L>0$ and let   $K>0$ be the constant of Lemma \ref{pq}. Then there exists a constant $C>0$ depending on $G$ such that   if $x$ and $y$ are  hyperbolic transformations  in $G$ whose axes are distinct and intersect, and whose translation lengths are bounded  above by $L$ then for each pair of positive integers $p$ and $q$, the $K$-quasigeodesic $\tg(x^p,y^q)$ satisfies  $\tg(x^p,y^g) \subset N_{C/2}(A_{g})$ and  $A_{g} \subset N_{C/2}(\tg(x^p,y^q))$.
\end{lemma}

\begin{proof} 
Denote by $d[p,q]$ the distance between $P$ (the point in $A_x \cap A_y$) and $A_g$.  Consider the region $\Lambda$ bounded by the axes $A_x$ and $A_y$ and  the arc of the circle of center $P$ and radius $d[p,q]$.  The area of $\Lambda$ equals $2 \theta_1 \sinh^2(d[p,q]/2)$. Also, $\Lambda$ is included in the triangle $T$, of area bounded by above by $\pi - \theta_1$ (see Figure~\ref{arc}). Hence,
$$
2 \sinh^2(d[p,q]/2) \le (\pi - \theta_1)/\theta_1 \le \pi / \delta.$$
Therefore, there exists a constant $C_1>0$ such that $d[p,q] \le C_1$ for all positive integers $p$ and $q$. Observe (Figure \ref{qgeo1}) the distance between any point in  $\tg(x^p,y^g)$ and $A_g$ smaller than that $d[p,q]$. This implies $\tg(x^p,y^g) \subset N_{C_1}(A_g)$.

\begin{figure}[htbp]
\begin{center}
\begin{pspicture}(0,1)(12,3)%


\psarc[linewidth=2pt,linecolor=darkgray](6,1){2.2}{145}{180}


\pspolygon[](2,1)(6,1)(3,3)


%

\rput(5,1.3){$\theta_1$}

\rput(1.25,1){$y^{-q}(P)$}
\rput(5.9,0.7){$P$}
\rput(2.4,3){$x^p(P)$}
\rput(4.9,2.6){$A_{g}$}
\psline[linecolor=gray,linewidth=2pt](3.432,0.6)(4.325,3)
 \end{pspicture}

\end{center}
\caption{The region $\Lambda$}\label{arc}
\end{figure}
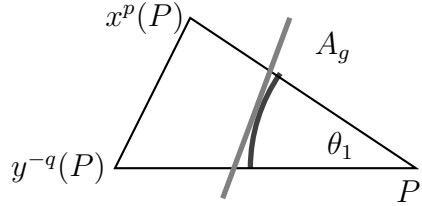

Denote by $R$ the intersection point of $A_{g}$ with $A_x$ and by $Q$ the intersection point of $A_{g}$  with $A_y$ (see Figure~\ref{qgeo1}).

Consider the triangle with vertices $P$, $Q$ and $R$. Triangles in the  hyperbolic plane $\H^2$ are $\ln(1+\sqrt{2})$-thin \cite[Fact 4, page 90]{hg}. In particular, the side of the triangle included in $A_{g}$  is at distance at most $\ln(1+\sqrt{2})$ of the union of the other two sides. 

By taking $C=2\max\{\ln(1+\sqrt{2}),C_1\}$  the desired result follows.

\end{proof}

\begin{lemma}\label{neigh} Let $x$ and $y$ be two hyperbolic transformations in $G$ whose axes intersect at a point $P$. Let $p$ and $q$ be positive integers such that $p\cdot \tax \ge 6KC$, where $K$ and $C$ are as in Lemmas \ref{pq} and \ref{pq1}. Denote by $I$ the segment of $A_x$ from $P$ to $x^p(P)$. 

Let $S$ and $R$ be the points in $A_x$ at distance $3KC$ of $P$ and $x^p P$ (Figure \ref{m6}.)

Let $s$ (resp. r) be the open half-plane bounded by the line  perpendicular to $A_x$ through $S$ (resp. $R$), containing the point $x^pP$ (resp. $P$). 

Set $U = s \cap r \cap N_C(I)$. 

Then $U$ contains an open subsegment $J$ of $I$ of length at least $p\cdot\tax-6KC$. Moreover, $U \subset N_C(I)$ and  $\mathrm{closure}(U) \cap N_{C}(L)=\emptyset$ for all maximal geodesic segments of $L$ of $\tg(x^p,y^g)$ distinct from $I$.
\end{lemma}
\begin{proof} 

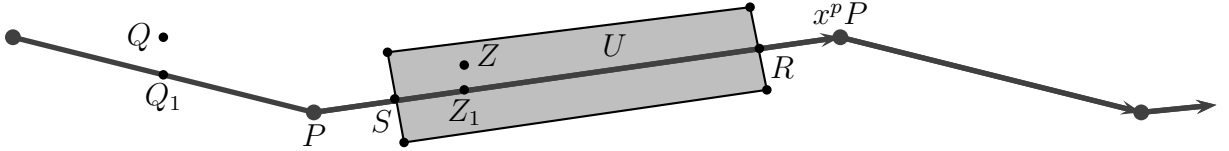
\begin{figure}[htbp]
\begin{pspicture}(0,-1)(16,1.6)


\pspolygon[fillcolor=lightgray,fillstyle=solid,showpoints=true](4.98,0.8)(5.2,-0.4)(10.023,0.3)(9.8,1.4)

\rput(8,0.9){$U$}
\rput{8}(4.5,0.0751){
}

\psline[linecolor=darkgray,linewidth=2pt,showpoints=true]{->}(0,1)(4,0)(11,1.)(15,0)(16,0.1)

\psline[linecolor=darkgray,linewidth=2pt]{->}(4,0)(11,1.)
\psline[linecolor=darkgray,linewidth=2pt]{->}(11,1.)(15,0)
\psline[linecolor=darkgray,linewidth=2pt]{->}(15,0)(16,0.1)

\rput(4,-0.3){$P$}
\rput(11,1.32){$x^pP$}

\rput(4.9,-0.12){$S$}\psdot(5.0762,0.18)
\rput(10.23,0.6){$R$}\psdot(9.924,0.85)

\psdot(6,0.3) \rput(6,-0.){$Z_1$}
\psdot(6,0.63) \rput(6.32,0.76){$Z$}

\psdot(2,1) \rput(1.68,1.){$Q$}
\psdot(2,0.5) \rput(2,0.2){$Q_1$}

 \end{pspicture}
 \caption{The quasi-geodesic}\label{m6}
\end{figure}
 Let $Q \in N_C(L)$, where  $L$ is a maximal segment of $\tg(x^p,y^q)$ different from $I$. Assume that  $Q$ and $P$ are in the same component of $N_C(\tg(x^p, y^q)) \setminus U$ (the proof is analogous in the other case). Let $Z$ be any point in $\mathrm{closure}(U)$ and let $Z_1$ be a point in $I \cap U$ at distance smaller than $C$ from $Z$.  Let $Q_1$ be a point in $L$ at distance smaller than $C$ of $Q$.
Then
$$
d(Q,Z) \ge d(Q_1,Z_1)-2C \ge d(P,S)/K-2C \ge C.
$$
Since $d(S,R) \ge p \cdot \tax-6KC$, one can take $J$ as the segment of $A_x$ from $S$ to $R$. This completes the proof of the lemma.

\end{proof}


We can (and will) assume without loss of generality that $K \ge 1$.

The following lemma is key to the paper.  
\begin{lemma}\label{noncanc first}  For each $L>0$ there exist a positive integer $p_0$ such that  for each pair of integers $p$ and $q$   satisfying  $ p \ge p_0$, and for each pair of hyperbolic transformations $x$ and $y$ (resp. $\xx$ and $\yy$)  whose axes  are distinct and intersect, and whose translation length is bounded by above by $L$, if $x^p y^q=\xx^p\yy^q$, then $\tg(x^p,y^q)=\tg(\xx^p,\yy^q)$.
\end{lemma}
\begin{proof}

We start by describing the two parts of the proof.
First, in the situation above, the two corresponding quasigeodesics are one in a $C$-neighborhood of the other. In particular, segments of one quasi-geodesic are in $C$-neighborhood of segments of the other quasi-geodesics.  By making the integer $p$ long enough, we obtain a "long" geodesic segment in a $C$-neighborhood of other geodesic segment. This implies that these two segment intersect in an interval. 

Second, we use the fact that the quasi-geodesics are constructed by translating two consecutive maximal segments by powers of $g$, to show if the two intersecting segments are distinct, an impossible figure is obtained.

Here are the details of the proof: For each finitely generated, discrete subgroup $G$ of $\mathrm{Isom}(\H)$, there exist a positive constant $\tau_0$ such that for each hyperbolic transformation $x \in G$, $\tau_x \ge \tau_0$ (see, for instance, \cite[Theorem 1.4.2]{Green})

Let $C$ and $K$ be as in Lemmas~\ref{pq} and \ref{pq1}. Let $M$ be the constant of Corollary \ref{smallint}. We will show that  $p_0= K(3M+10C)/ \tau_0$ gives the desired conclusion.

Since $x^p y^q=\xx^p\yy^q$, $A_{x^p y^q}=A_{\xx^p\yy^q}$.
  By Lemma \ref{pq1}, $$\tg(\xx^p,\yy^q) \subset N_{C/2}(A_{g}) \subset N_C(\tg(x^p,y^q)).$$ Let $U$ and $J$ be the neighborhood and  the segment given by   Lemma~\ref{neigh} respectively, so $J \subset U$, $J\subset I\subset \tg(x^p,y^q)$ and the length of $J$ is at least $p \tax - 6KC$. 

Observe that $\tg(\xx^p,\yy^p)$ must intersect $U$, otherwise $\tg(\xx^p,\yy^q)$ is included in  $N_C(\tg(x^p,y^q)\sm J)$, which has two components. Furthermore,   $\tg(\xx^p,\yy^p)$ must intersect both the components, contradicting the fact that $\tg(\xx^p,\yy^q)$ is connected. By Lemma \ref{neigh}, $N_C(L) \cap \mathrm{closure}(U) = \emptyset$ for all maximal segments  $L$ of $\tg(x^p,y^q)$ distinct from $I$. 
Hence, $\tg(\xx^p,\yy^q)$ does not intersect the intersection of the boundary of $U$ with the boundary of  $N_C(\tg(x^p,y^q))$. 

By hypothesis, the length of $J$ is at least $p\tax -  6KC$ so it is at least $3KM+4KC$

\begin{figure}[htbp]
\begin{center}
\begin{pspicture}(0,0.8)(12,3)%

\pspolygon[fillstyle=solid,fillcolor=lightgray](1,0.5)(10,0.3)(10,3)(0.3,3)

\psline(0.7,1.687642)(2,1.5)(5,2.5)(8,1.6)(10,01.8)

\rput(0.25,1.9){$2C$}
\rput(10.3,1.74){$2C$}
\rput(9.684,2.68){$U$}
\rput(6,1.5){$\tg(\xx^p,\yy^p)$}
 \end{pspicture}

\end{center}
\caption{The intersection of  neighborhood $U$ of $J$ with $\tg(\xx^p,\yy^q)$}\label{arc1}
\end{figure}
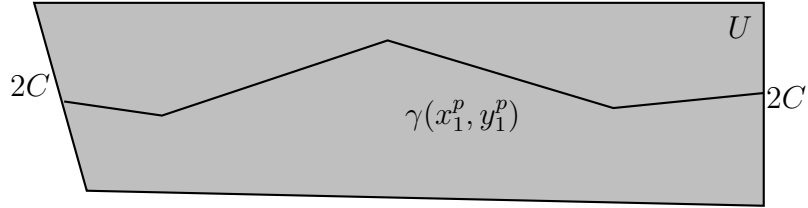

Thus, the components of the set $U \cap \tg(\xx^p,\yy^q)$ are piecewise linear curves starting and ending at the sides of $U$ of length $2C$ (see Figure~\ref{arc1}). Let $\beta$ be one of these components.  We claim that $\beta$ contains a segment $l$ of length greater than $M$. 
Indeed, if $\beta$ contains three or more vertices of $\tg(\xx^p,\yy^q)$ then one segment of $\beta$ is a maximal segment of  $\tg(\xx^p,\yy^q)$ included in  a translate of $\xx^p$. Therefore, it must have length at least $p_0 \tax$.  Otherwise, $\beta$ consists in at most three segments. Denote by $m$ the length of the longest of these segments. By the triangular inequality,
$$
K(3 M+4C) \le p \tax - 6KC \le 3m+4C.  
$$
Since $K >1$, $m>M$.  Thus the claim is proved.

The segment $l$ of $\beta$ of length at least $M$ is included in some segment $I'$ of $\tg(\xx^p,\yy^q)$. Thus  $I' \cap N_C(J)$ contains a segment longer than $M$.   By Lemma~\ref{smallint},  $I'$ intersects $I$ in a subsegment.  This concludes the first part of the proof. We will show that the assumption  $I \ne I'$ leads to a contradiction.

If $I\neq I'$ by interchanging the roles of $I$ and $I'$ if necessary, we can assume that there is a vertex $v$ of $I$ which is not in $I'$. Let $v'$ be the vertex of $I'$ closest to $v$. Denote by  $L$ (resp. $L'$) be the  maximal segment of   $\tg(x^p,y^q)$ (resp. $\tg(\xx^p,\yy^q)$)  so that $I$ and $L$  (resp. $I'$ and $L'$) are adjacent and intersect in $v$ (resp. $v'$).

Recall that $\tg(x^p,y^q)$  (resp. $\tg(\xx^p,\yy^q)$) is constructed by taking two consecutive maximal segments and translating them by powers of $g$. To simplify the notation, we write $g=x^py^q$. The segment adjacent to $L$ (resp. $L'$) different from $I$ (resp. $I'$) is $g(I)$  (resp. $g(I')$).
Denote by $u$ (resp. $u'$) the other vertex of $I$ (resp. $I'$) .
Note that $v$ and $g(u)$ (resp. $v'$ and  $g(u')$) are the vertices of $L$ (resp. $L'$).
 
 Suppose first that $u$ is in $I'$.
  By Lemma~\ref{pq}, the angles $u,v,g(u)$ and $v, g(u), g(v)$ are congruent. 
Hence there is a convex quadrilateral with vertices $v, v',  g(u),  g(u')$, Figure~\ref{qg2}. By Lemma~\ref{pq}, the sum of the interior angles of this quadrilateral is $2\pi$, a contradiction in hyperbolic geometry. This implies that $u$ is not in $I'$.

\begin{figure}[htbp]
\begin{center}
\begin{pspicture}(14,3)%

\rput(2,0.5){
\rput(6.2, 1.8){$v$}

\rput(0,2.32){$I'$}
\rput(5,2.32){$I$}
\rput(2.65,1.2){$L$}
\rput(4.975,1.3){$L'$}

\psdot[dotsize=0.25](9,0)
\rput(9, -0.4){$g(v)$}

\psdot[dotsize=0.25](7,0)
\rput(7, -0.4){$g(v')$}

\psdot[dotsize=0.25](6,2)

\rput(1, 1.7){$u$}
\psdot[dotsize=0.25](1,2)

\psline(1,2)(6,2)(4,0)(9,0)

\rput(3.368, 1.8){$v'$}
\psdot[dotsize=0.25](4,2)

\rput(-1, 1.7){$u'$}
\psdot[dotsize=0.25](-1,2)

\psline[linestyle=dashed,linewidth=0.1](-1,2)(4,2)(2,0)(7,0)

\psdot[dotsize=0.25](2,0)
\rput(2., -0.4){$g(u')$}

\psdot[dotsize=0.25](4,0)
\rput(4, -0.4){$g(u)$}

\rput(8, 0.4){$g(I)$}

}
 \end{pspicture}

\end{center}

\caption{Length of $I$ equals the length of $I'$}\label{qg2}
\end{figure}
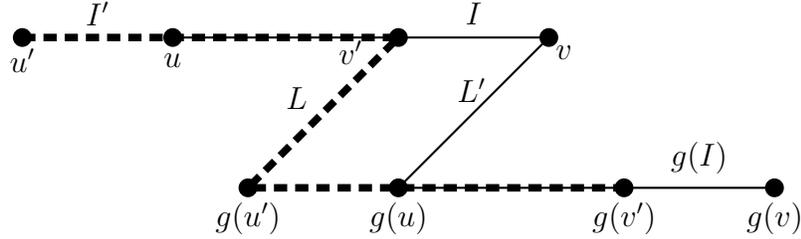 

Denote by $l$ the geodesic through $v$ and $g(u)$. 
By Lemma~\ref{pq}, the angles $u, v, g(u)$ and $v, g(u), g(v)$ are congruent.
This implies that $u$ and $g(v)$ are in different sides of $l$. On the other hand, $u$ and $v'$ (resp. $g(v)$ and $g(u')$) are on the same side of $l$. 
Then $vÕ$ and $g(uÕ)$ are on different sides of $l$. Hence $L$ intersects $L'$ and  
 the quasi-geodesics are arranged as in Figure~\ref{qg3}.

In particular, the segments $L$ and $L'$ intersect at a point $z$. The triangles with vertices $z, v', v$ and $z, g(u), g(u')$ have congruent corresponding angles. Hence, these two triangles  are congruent. Thus, $z$ is the middle point of $L$ and also, of $L'$.  Since the segments with vertices $u,u'$ and $g(u), g(u')$ are congruent, the segments with vertices $u, u'$ and $v, v'$ are congruent.

Denote by $w$ the middle point of $I$. Observe that $w$ is also the middle point of $I'$. The length of the arc of $\tg(\xx^p,\yy^q)$ from $w$ to $z$ equals $(\tax+\tay)/2$. 
Also, the length of the arc of $\tg(x^p,y^q)$ from $w$ to $z$ equals $(\tax+\tay)/2$. By the triangle inequality, this is impossible. Thus we conclude that $v=v'$.

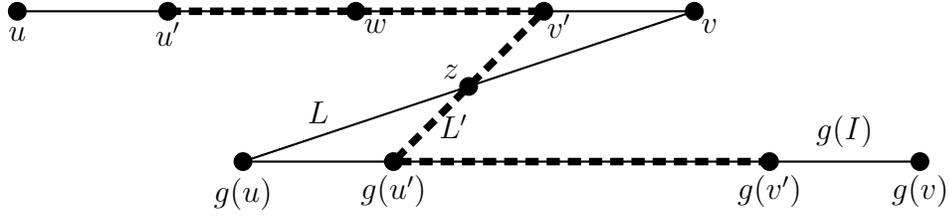
\begin{figure}[htbp]
\begin{center}
\begin{pspicture}(14,3)%

\rput(2, 0.5){

\rput(3,0.65){$L$}
\rput(4.8,0.465){$L'$}

\rput(3.765, 1.8){$w$}
\psdot[dotsize=0.25](3.5,2)

\rput(4.765, 1.2){$z$}
\psdot[dotsize=0.25](5,1)

\rput(6.2, 1.8){$v'$}
\psdot[dotsize=0.25](6,2)

\rput(1, 1.7){$u'$}
\psdot[dotsize=0.25](1,2)

\psline(-1,2)(8,2)(2,0)(11,0)

\rput(8.2, 1.8){$v$}
\psdot[dotsize=0.25](8,2)

\rput(-1, 1.7){$u$}
\psdot[dotsize=0.25](-1,2)

\psline[linestyle=dashed,linewidth=0.1](1,2)(6,2)(4,0)(9,0)
\psdot[dotsize=0.25](2,0)
\rput(2., -0.43){$g(u)$}

\psdot[dotsize=0.25](4,0)
\rput(4., -0.4){$g(u')$}

\psdot[dotsize=0.25](9,0)
\rput(9., -0.4){$g(v')$}

\psdot[dotsize=0.25](11,0)
\rput(11., -0.4){$g(v)$}

\rput(10, 0.4){$g(I)$}

}

 \end{pspicture}

\end{center}

\caption{ Length of $I$ is larger than length of $I'$}\label{qg3}
\end{figure}

Using the above arguments we  prove that  $L=L'$. Continuing this argument we see that  the quasi-geodesics $\tg(x^p,y^q)$ and $\tg(\xx^p,\yy^q)$ coincide.

\end{proof}

\begin{theorem}\label{noncanc}  For each $L>0$ there exist a positive integer $p_0$ such that  if $p$ and $q$ are integers   satisfying  $ p \ge p_0$, and  $x$ and $y$ (resp. $\xx$ and $\yy$) are hyperbolic transformations whose axes  are distinct and intersect,  whose translation length is bounded by above by $L$,  $p\tax \ne q \tay$ 
and $x^p y^q=\xx^p\yy^q$, then there exists  $g \in G$, such that  $\xx=x^g$ and $\yy=y^g$.
\end{theorem}
\begin{proof} Since $x^p y^q=\xx^p\yy^q$, $A_{x^p y^q}=A_{\xx^p\yy^q}$. Moreover, both axes are oriented in the same direction.

 If $p_0$ is the positive integer given by  Lemma \ref{noncanc first}, then $\tg(x^p,y^q)=\tg(\xx^p,\yy^q)$. This implies that there exist $u$ and $v$ in $G$ such that that   one of the following holds 
  
 \begin{numlist}
 \item $\xx^p=(x^p)^u=(x^p)^{h^n}$ and $\yy^q=(y^q)^v=(y^q)^{h^n}$.
 \item $\xx^p=(x^p)^u=(y^q)^{h^{n+1}}$ and $\yy^q=(y^q)^v=(x^p)^{h^n}$.
 \end{numlist}

Since (2) implies that $p \tax = q \tay$, the result follows by taking $g=h^n$.
 \end{proof}

\section{Proof of the main theorem}\label{main}

An element $z$ in $\Z[\cc]$ can be uniquely represented as a sum $\sum_{i=1}^k n_i\la x_i\ra$ so that the conjugacy classes $\la x_i\ra$ are all distinct and the integers $n_i$ are non-zero. We define the \emph{Manhattan norm of $z$} by
$$M\left( \sum_{i=1}^k n_i\la x_i\ra\right)=\sum_{i=1}^k \left\vert n_i\right\vert.$$

We are now in a position to prove our Main Theorem. Denote by $X_p$ and $Y_q$ the cyclic groups generated by $x^p$ and $y^q$ respectively. Note that by definition 
$$[\la x^p\ra,\la y^q \ra]=\sum_{_{X_pbY_q \in I(x^p,y^q)}} \inu(x^p,(y^p)^b)\conjj{ x^p(y^p)^b }.$$

Our first step is to collate terms in this expression. There is a natural quotient map from $X_p\backslash G/Y_q$ to $X\backslash G/Y$, mapping $X_p\backslash g /Y_q$  to $X\backslash g /Y$.  Observe that $\inu(x^p,(y^p)^b)=\inu(x,y^b)$. Further observe that if $X_p\backslash g /Y_q$  and $X_p\backslash g'/Y_q$ map to the same element in $XgY$, then $\la x^p(y^p)^{g} \ra=\la x^p (y^p)^{g'}\ra=\la x^p(y^{g'})^q\ra$. 
The lemma below follows by grouping terms corresponding to their images in $I(x,y)$.

\begin{lemma}
We have
$$[\la x^p\ra,\la y^q \ra]=pq\left(\sum_{XbY\in I(x,y)} \inu(x,y^b)\conjj{ x^p(y^b)^q}\right) .$$
\end{lemma}

We are now ready to prove our main result.

\begin{proclama}{Main Theorem} Let $G$ be a finitely generated, discrete group of $\mathrm{Isom}(\H)$ and let $L>0$. There exists $p_0$ such that if $p$ and $q$ are integers at least one  of which is larger than $p_0$ then the following holds:
\begin{numlist}
\item If $x$ and $y$ are non-conjugate hyperbolic transformations in $G$, with translation length bounded by above by $L$ such that $p\tau(x)\neq q\tau(y)$ then $\frac{M[x^p,y^q]}{p\cdot q}$ equals the geometric intersection number of $x$ and $y$.
\item If $p \ne q$, and  $x$ is a  hyperbolic transformation in $G$,  not a proper power, and has translation length bounded by above by $L$ then  $\frac{M[x^p,x^q]}{2 \cdot p \cdot q}$ equals the geometric self-intersection number of $x$.
\end{numlist}
\end{proclama}
\begin{proof}
Interchanging $x$ and $y$ if necessary, we can assume that $p\geq p_0$.

Suppose that $\la x^p(y^b)^q\ra = \la x^p(y^{b'})^q\ra$. Then  for some $h \in G$,
$$x^p(y^b)^q=(x^p (y^{b'})^q)^h=(x^p)^h (y^q)^{hb'}.$$
By Theorem~\ref{noncanc}, there is an element $g$ that conjugates $x$ to $x^g$ and $y^b$ to $y^{gb'}$. In particular, the signs $\inu(x,y^b)$ and $\inu(x^g, y^{gb'})$ coincide, so there is no cancellation. This concludes the proof.
\end{proof}

\begin{figure}
\begin{center}

\psset{xunit=0.3cm,yunit=0.3cm,algebraic=true,dotstyle=o,dotsize=3pt 0,linewidth=0.8pt,arrowsize=3pt 2,arrowinset=0.25}
\begin{pspicture*}(-1.51,-1.69)(20.75,13.43)
\psaxes[labelFontSize=\scriptstyle,xAxis=true,yAxis=true,Dx=2,Dy=2,ticksize=-2pt 0,subticks=2]{->}(0,0)(-1.51,-1.69)(20.75,13.43)
\psplot[linestyle=dashed,dash=2pt 2pt]{0}{20.75}{(-0--2*x)/3}
\begin{scriptsize}
\psdots[dotsize=5pt 0](3,2)
\psdots[dotsize=5pt 0](3,3)
\psdots[dotsize=5pt 0](6,4)
\psdots[dotsize=5pt 0](9,6)
\psdots[dotsize=5pt 0](12,8)
\psdots[dotsize=5pt 0](15,10)
\psdots[dotsize=5pt 0](1,2)
\psdots[dotsize=5pt 0](1,3)
\psdots[dotsize=5pt 0](2,2)
\psdots[dotsize=5pt 0](2,1)
\psdots[dotsize=5pt 0](1,1)
\psdots[dotsize=5pt 0](3,1)
\psdots[dotsize=5pt 0](2,3)
\psdots[dotsize=5pt 0](18,12)
\psdots[dotsize=5pt 0](4,1)
\psdots[dotsize=5pt 0](4,2)
\psdots[dotsize=5pt 0](4,3)
\psdots[dotsize=5pt 0](4,4)
\psdots[dotsize=5pt 0](3,4)
\psdots[dotsize=5pt 0](2,4)
\psdots[dotsize=5pt 0](1,4)
\end{scriptsize}
\end{pspicture*}

\end{center} 
\caption{Values of $p$ and $q$ in the Main Theorem, $p_0=5$, $p/q\ne 3/2$, $p$ in $x$-axis}
\end{figure}
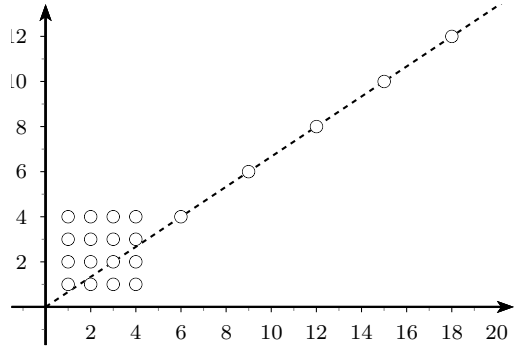


\end{document}